\documentclass[reqno,a4paper,12pt]{article}

\usepackage[english]{babel}
\usepackage{amsmath,amssymb,amsfonts,amsthm,enumerate,mathrsfs}

\newcommand{\diam}{\operatorname{diam}}
\newcommand{\BB}{\mathbb{B}}
\newcommand{\Q}{\mathbb{Q}}
\newcommand{\N}{\mathbb{N}}
\newcommand{\R}{\mathbb{R}}

\newcommand{\Z}{\mathbb{Z}}
\newcommand{\Haus}[1]{\mathcal{H}^{#1}}     % Misura di Hausdorff
\newcommand{\Leb}[1]{\mathcal{L}^{#1}}      % Misura di Lebesgue
\newcommand{\T}{\mathbb T}

\newcommand{\F}{\mathcal F}
\newcommand{\GG}{\mathcal G}

\newcommand{\Aub}{\mathcal{A}}
\newcommand{\Aubry}{\mathcal{A}}
\newcommand{\torus}{\mathbb{T}}

\renewcommand{\a}{\alpha}

\renewcommand{\d}{\delta}

\newcommand{\e}{\varepsilon}
\newcommand{\g}{\gamma}

\renewcommand{\O}{\Omega}

\renewcommand{\i}{\infty}

\newcommand{\ov}{\overline}

\newcommand\INT[1]{\mathaccent'27{#1}}

\title{On the Hausdorff Dimension of the Mather Quotient}

\author{ Albert Fathi
         \thanks{UMPA, ENS Lyon,
         46 All\'ee d'Italie,
         69007 Lyon, France.
	 \textsf{e-mail: albert.fathi@umpa.ens-lyon.fr}},\ \ \
         Alessio Figalli
         \thanks{Universit\'e de Nice-Sophia
    Antipolis, Parc
    Valrose, 06100 Nice, France.
         \textsf{e-mail: figalli@unice.fr}},\ \ \
	 Ludovic Rifford
	 \thanks{Universit\'e de Nice-Sophia
    Antipolis, Parc
    Valrose, 06100 Nice, France.
	 \textsf{e-mail: rifford@unice.fr} } }

\date{8 November, 2007}

\newtheorem{theorem}{Theorem}[section]
\newtheorem{corollary}[theorem]{Corollary}
\newtheorem{proposition}[theorem]{Proposition}
\newtheorem{lemma}[theorem]{Lemma}
\theoremstyle{remark}
\newtheorem{remark}[theorem]{\bf Remark}
\theoremstyle{definition}
\newtheorem{definition}[theorem]{Definition}

\theoremstyle{definition}

\theoremstyle{remark}

\makeindex
\begin{document}

\maketitle

\begin{abstract}
Under appropriate assumptions on the dimension of the ambient manifold and the regularity of the Hamiltonian, we show that the Mather quotient  is small in term of Hausdorff dimension. Then, we present applications in dynamics.
\end{abstract}

\section{Introduction}

Let $M$ be a smooth manifold without boundary. We denote by $TM$ the tangent bundle and by $\pi:TM \rightarrow M$ the canonical projection. A point in $TM$ will be denoted by $(x,v)$ with $x\in M$ and $v\in T_xM=\pi^{-1}(x)$. In the same way a point of the cotangent bundle  $T^*M$ will be denoted by $(x,p)$ with $x\in M$ and $p\in T_x^*M$ a linear form on the vector space $T_xM$. We will suppose that $g$ is a complete Riemannian metric on $M$. For $v \in T_xM$, the norm $\|v \|_x$ is $g_x(v,v)^{1/2}$. We will denote by $\|\cdot \|_x$ the dual norm on $T^*M$. Moreover, for every pair $x,y \in M$, $d(x,y)$ will denote the Riemannian distance from $x$ to $y$.

We will assume in the whole paper that $H:T^*M \rightarrow \R$ is an Hamiltonian of class $C^{k,{\alpha}}$, with $ k\geq 2, \alpha \in [0,1]$, which satisfies the three following conditions:
\begin{enumerate}
\item[(H1)]
 \textbf{{\rm C}$^2$-strict convexity:}
$\forall (x,p) \in T^*M$, the second derivative along the fibers ${\partial^2 H/{\partial p^2}(x,p)}$ is strictly positive definite;
\item[(H2)]
\textbf{uniform superlinearity:}
for every $K \geq 0$ there exists a finite constant $C(K)$  such that
$$
\forall (x,p) \in T^*M, \quad H(x,p) \geq K \|p\|_x + C(K);
$$
\item[(H3)]
\textbf{uniform boundedness in the fibers:}
for every $R \geq 0$, we have
$$
\sup_{x \in M} \{H(x,p) \mid \|p\|_x  \leq R\} <+\i.
$$
\end{enumerate}

By the Weak KAM Theorem we know that, under the above conditions, there is $c(H)\in \R$ such that the Hamilton-Jacobi equation 
\begin{equation*}
H(x,d_xu) =c \tag{HJ$_c$}
\end{equation*}
admits a global viscosity solution $u:M \rightarrow \R$ for $c=c(H)$ and does not admit such solution for $c<c(H)$, see \cite{lpv87,fathi97,contreras01,fathibook,fm}. In fact, for $c<c(H)$, the Hamilton-Jacobi equation 
does not admit any viscosity subsolution (for the theory of viscosity solutions, we refer the reader to the monographs \cite{bc,bar,fathibook}). Moreover, if $M$ is assumed 
to be compact, then $c(H)$ is the only value of $c$ for which the 
Hamilton-Jacobi equation above admits a viscosity solution. 
The constant $c(H)$ is called the \textit{critical value}, or 
the \textit{Ma\~n\'e critical value} of $H$. In the sequel, a viscosity solution $u:M\rightarrow \R$ of $H(x,d_xu)=c(H)$ will be called a
 \textit{critical viscosity solution} or a \textit{weak KAM solution}, 
 while a viscosity subsolution $u$ of $H(x,d_xu)=c(H)$ will be called 
 a \textit{critical viscosity subsolution} (or \textit{critical subsolution} 
 if $u$ is at least {\rm C}$^1$). 

The Lagrangian $L:TM \rightarrow \R$ associated to the Hamiltonian $H$ is defined by 
$$
\forall (x,v) \in TM, \quad L(x,v) = \max _{p\in T_x^*M} \left\{  p(v) -H(x,p) \right\}.
$$
Since $H$ is of class {\rm C}$^k$, with $k\geq 2$, and satisfies the three conditions (H1)-(H3), it is well-known
(see for instance \cite{fathibook} or \cite[Lemma 2.1]{fm})) that $L$ is finite everywhere of class {\rm C}$^k$, and is a Tonelli Lagrangian, i.e. satisfies the analogous of conditions (H1)-(H3). Moreover, the Hamiltonian $H$ can be recovered from $L$ by
$$
\forall (x,p) \in T_x^*M, \quad H(x,p) = \max_{v\in T_xM} \left\{ p(v) -L(x,v) \right\}.
$$
Therefore the following inequality is always satisfied 
$$
p(v) \leq L(x,v) + H(x,p).
$$
This inequality is called the Fenchel inequality. Moreover, due to the strict convexity of $L$, we have equality in the Fenchel inequality if and only if 
$$
(x,p) =\mathcal{L}(x,v),
$$
where $\mathcal{L}:TM \rightarrow T^*M$ denotes the Legendre transform defined as
$$
\mathcal{L}(x,v) = \left( x,\frac{\partial L}{\partial v} (x,v)\right).
$$
Under our assumption $\mathcal L$ is a diffeomorphism of class at least $C^{1}$.
We will denote by $\phi_t^L$ the Euler-Lagrange flow of $L$, and by $X_L$ 
the vector field on $TM$ that generates the flow $\phi_t^L$. If we denote by 
$\phi_t^H$ the Hamiltonian flow of $H$ on $T^*M$, then as is well-known, 
see for example 
\cite{fathibook}, this flow $\phi^H_t$ is conjugate to $\phi^L_t$ by the Legendre transform $\mathcal{L}$.

As done by Mather in \cite{mather93}, it is convenient to introduce for $t>0$ fixed, the function
$h_t:M \times M \rightarrow \R$ defined by 
$$
\forall x,y \in M, \quad h_t(x,y) = \inf \int_0^t L(\gamma(s), \dot{\gamma}(s))ds,
$$
where the infimum is taken over all the  absolutely continuous paths $\gamma :[0,t] \rightarrow M$ with $\gamma(0)=x$ and
$\gamma(t)=y$. The \textit{Peierls barrier} is the function $h:M \times M \rightarrow \R$ defined by 
$$
h(x,y) = \liminf_{t\rightarrow \infty} \left\{ h_t(x,y) + c(H) t\right\}.
$$
It is clear that this function satisfies 
\begin{align*}
\forall x,y,z \in M, \quad & h(x,z) \leq h(x,y) + h_t(y,z) + c(H)t \\
& h(x,z) \leq h_t(x,y) + c(H)t + h(y,z),
\end{align*}
and therefore it also satisfies the triangle inequality
$$
\forall x,y,z \in M, \quad h(x,z) \leq h(x,y) + h(y,z).
$$
Moreover, given a weak KAM solution $u$, we have
$$
\forall x,y \in M,\quad u(y)-u(x) \leq h(x,y).
$$
In particular, we have $h>-\infty$ everywhere. It follows, from the triangle inequality, that the function $h$ is either identically $+\infty$ or it is finite everywhere.
If $M$ is compact, $h$ is finite everywhere. In addition,  if $h$ is finite, then for each $x\in M$ the function $h_x(\cdot)=h(x,\cdot)$ is a critical viscosity solution (see \cite{fathibook} or \cite{fs04}). The \textit{projected Aubry set} $\mathcal{A}$ is defined by 
$$
\mathcal{A} =\{ x \in M \mid \ h(x,x) =0\}.
$$
Following Mather, see   \cite[page 1370]{mather93}, we symmetrize $h$ to define the function $\delta_M :M \times M\rightarrow \R$
by 
$$
\forall x,y \in M, \quad  \delta_M (x,y) = h(x,y) + h(y,x).
$$

Since $h$ satisfies the triangle inequality and $h(x,x)\geq 0$ everywhere, the function $\delta_M$  is symmetric, everywhere nonnegative and satisfies the triangle inequality.
The restriction $\delta_M :\mathcal{A} \times \mathcal{A}\rightarrow \R$ is
a genuine semi-distance on the projected Aubry set. We will call this function $\delta_M$ the \textit{Mather semi-distance} (even when we consider it on $M$ rather than on  $\mathcal{A}$). We define the \textit{Mather quotient} $(\mathcal{A}_M,\delta_M)$ to be the metric space obtained by identifying two points  $x, y\in \mathcal{A}$ 
if their semi-distance $\delta_M(x,y)$ vanishes. When we consider $\delta_M$ on the quotient space $\mathcal{A}_M$ we will call it the \textit{Mather distance}. 

In \cite{mather04}, Mather formulated the following problem:\\

\noindent \textbf{Mather's Problem}. If $L$ is {\rm C}$^\i$, is the set $\Aub_M$ totally disconnected for the topology of $\delta_M$, i.e. is each connected component of $\Aub_M$ reduced to a single  point?\\

In \cite{mather03}, Mather brought a positive answer to that problem in low dimension. More precisely, he proved that  if $M$ has dimension two, or if the Lagrangian is the kinetic energy associated to a Riemannian metric on $M$ in  dimension $\leq 3$, then the quotient Aubry set is totally disconnected. Notice that one can easily show that for a dense set of Hamiltonians, the set $(\mathcal{A}_M,\delta_M)$ is reduced to one point. Mather mentioned in \cite[page 1668]{mather04} that it would be even more interesting to be able to prove that the quotient Aubry set has vanishing one-dimensional Hausdorff measure, because this implies the upper semi-continuity of the mapping $H \mapsto \mathcal{A}$. He also stated that for Arnold's diffusion a result generic in the Lagrangian but true for every cohomology class was more relevant. This was obtained recently by Bernard and Contreras \cite{bercon}. 

The aim of the present paper is to show that the vanishing of the one-dimensional Hausdorff measure of the Mather quotient   is satisfied under various assumptions. Let us state our results. 

\begin{theorem}
\label{THM1}
If $\dim M =1,2$ and $H$ of class {\rm C}$^2$ or  $\dim M=3$ and $H$ of class {\rm C}$^{k,1}$ with $k\geq 3$, then the Mather quotient $(\Aub_M,\delta_M)$  has vanishing one-dimensional Hausdorff measure. 
\end{theorem}

Above the projected Aubry $\mathcal A$, there is a compact subset  $\tilde{\mathcal{A}} \subset TM$ called the Aubry set (see Section \ref{DefAubry}). The projection $\pi: TM\to M$ induces a homeomorphism $\pi_{|\tilde{\mathcal{A}}}$ from $\tilde{\mathcal{A}}$ onto ${\mathcal{A}}$ (whose inverse is Lipschitz by a theorem due to Mather). 
The Aubry set can be defined as the set of $(x,v) \in TM$ such that $x\in \mathcal{A}$ 
and $v$ is the unique element in $T_xM$ such that $d_xu ={\partial L}/{\partial v}(x,v)$ for any 
critical viscosity subsolution $u$. The Aubry set is invariant under the Euler-Lagrange
flow $\phi_t^L:TM\to TM$. Therefore, for each $x\in \mathcal{A}$, 
 there is only one orbit of $\phi_t^L$ in $\tilde{\mathcal{A}}$ whose projection passes through $x$.  We define the \textit{stationary Aubry set} $\tilde{\mathcal{A}}^0\subset \tilde{\mathcal{A}}$ as the set of points in $\tilde{\mathcal{A}}$ which are  fixed points of the Euler-Lagrange flow $\phi_t(x,v)$, i.e.
 $$
\tilde{\mathcal{A}}^0=\{(x,v) \in \tilde{\mathcal{A}} \mid \forall t\in \R, \phi_t^L(x,v)=(x,v)\}.
$$
 In fact, see Proposition \ref{caracAub0}, it can be shown, that $\tilde{\mathcal{A}}^0$ is the intersection of $\tilde{\mathcal{A}}$ with the zero section of $TM$
 $$\tilde{\mathcal{A}}^0=\{(x,0) \mid (x,0)\in\tilde{\mathcal{A}}\}.$$

We define the  \textit{projected stationary Aubry set} $ \mathcal{A}^0$ as the projection on $M$ of  $\tilde{\mathcal{A}}^0$
$$ \mathcal{A}^0=\{x\mid (x,0)\in\tilde{\mathcal{A}}\}.$$
At the very end of his paper \cite{mather03}, Mather noticed that the argument he used in the case where $L$ is a kinetic energy in dimension 3 proves the total disconnectedness of the quotient Aubry set  in dimension 3 as long as $\mathcal{A}_M^0$ is empty. In fact, if we consider the restriction of $\delta_M$ to $ \mathcal{A}^0$, we have the following result on the quotient metric space  $(\mathcal{A}_M^0,\delta_M)$.

\begin{theorem}
\label{THM2}
Suppose that $L$ is at least {\rm C}$^2$, and that the restriction $x\mapsto L(x,0)$ of $L$ to the zero section of $TM$ is of class {\rm C}$^{k,1}$. Then $(\mathcal{A}_M^0,\delta_M)$ has vanishing Hausdorff measure in dimension  $2\dim M/(k+3)$. In particular, if $k \geq 2\dim M-3$ then $\Haus1(\mathcal{A}_M^0,\delta_M)=0$, and if $x\mapsto L(x,0)$ is {\rm C}$^{\infty}$ then $(\mathcal{A}_M^0,\delta_M)$ has zero Hausdorff dimension.
\end{theorem}

As a corollary, we have the following result which was more or less already mentioned by Mather in \cite[\S 19 page 1722]{mather04}, and proved by  
Sorrentino  \cite{sorrentino06}.

\begin{corollary}\label{COR}
Assume that $H$ is of class {\rm C}$^2$ and that its associated Lagrangian $L$ satisfies the following conditions: 
\begin{enumerate}
\item
$\forall x \in M,  \quad \min_{v \in T_xM} L(x,v)=L(x,0)$;
\item the mapping $x \in M \mapsto L(x,0)$ is of class  {\rm C}$^{l,1}(M)$ with $l\geq 1$.
\end{enumerate}
If   $\dim M=1,2$, or $\dim M \geq 3$ and  $l \geq 2\dim M-3$, then $(\mathcal{A}_M,\delta_M)$ is totally disconnected. In particular, if  $L(x,v)=\frac{1}{2}\lVert v \rVert_x^2 - V(x)$, with $V \in {\rm C}^{l,1}(M)$ and $l \geq 2\dim M-3$ {\rm (}$V \in {\rm C}^2(M)$ if $\dim M =1,2${\rm )},
then $(\mathcal{A}_M,\delta_M)$ is totally disconnected.
\end{corollary}

Since $\mathcal{A}^0$ is the projection of the subset $\tilde{\mathcal{A}}^0\subset \tilde{\mathcal{A}}$  consisting of points in $\tilde{\mathcal{A}}$ which are fixed under the the Euler-Lagrange flow $\phi_t^L$, it is  natural to consider $\mathcal{A}^p$ the set of $x\in \mathcal{A}$ which are projection of a point $(x,v)\in \tilde{\mathcal{A}}$ whose orbit  under the the Euler-Lagrange flow $\phi_t^L$ is periodic with strictly positive period. We call this set the \textit{projected periodic Aubry set}. We have the following result:

\begin{theorem}
\label{THM3}
If $\dim M\geq 2$ and  $H$ of class {\rm C}$^{k,1}$ with $k\geq 2$, then $(\mathcal{A}_M^p,\delta_M)$ has vanishing Hausdorff measure in dimension $8\dim M/(k+8)$.  In particular, if $k\geq 8\dim M -8$ then $\Haus1(\mathcal{A}_M^p,\delta_M)=0$, and if $H$ is {\rm C}$^{\infty}$ then $(\mathcal{A}_M^p,\delta_M)$ has zero Hausdorff dimension.
\end{theorem}

In the case of compact surfaces, using the finiteness of exceptional minimal sets of flows, we have: 

\begin{theorem}
 \label{THMSURF}
If $M$ is a compact surface of class {\rm C}$^\infty$ and $H$ is of class {\rm C}$^{\infty}$, then $(\Aub_M,\delta_M)$ has zero Hausdorff dimension.
\end{theorem}

In the last section, we present applications in dynamic whose Theorem \ref{TheoremePrincipal2} below is a corollary. If $X$ is a {\rm C}$^k$ vector field on $M$, with $k\geq 2$, the Ma\~n\'e Lagrangian $L_X:TM\to\R$ associated to $X$ is defined by
$$
L_X(x,v)=\frac12\lVert v-X(x)\rVert_x^2, \quad \forall (x,v)\in TM.
$$
We will denote by $\mathcal{A}_X$ the projected Aubry set of the Lagrangian $L_X$.

The first author has raised the following problem, compare with the list of questions\\
{\tt http://www.aimath.org/WWN/dynpde/articles/html/20a/}.\\

\noindent \textbf{Problem}. Let $L_X:TM\to\R$ be the Ma\~n\'e Lagrangian associated to the {\rm C}$^k$ vector field $X$ ($k \geq 2$) on the compact connected manifold $M$. 
\begin{itemize} 
\item[(1)] Is the set of chain-recurrent points of the flow of $X$ on $M$ equal to the projected Aubry set $\mathcal{A}_X$?
\item[(2)] Give a condition on the dynamics of $X$ that insures that the only weak KAM solutions are the constants.\\
\end{itemize}

The theorems obtained in the first part of the paper together with the applications in dynamics developed in Section 6 give an answer to this question when $\dim M\leq 3$.

\begin{theorem}
\label{TheoremePrincipal2}
Let $X$ be a {\rm C}$^k$ vector field, with $k\geq 2$, on the compact connected {\rm C}$^\infty$ manifold $M$. Assume that one of the conditions hold:
\begin{itemize}
\item[(1)] The dimension of $M$ is $1$ or $2$.
\item[(2)] The dimension of $M$ is $3$, and the vector field $X$ never vanishes.
\item[(3)] The dimension of $M$ is $3$, and $X$ is of class {\rm C}$^{3,1}$.
\end{itemize}
Then the projected Aubry set $\mathcal{A}_X$ of the  Ma\~n\'e Lagrangian $L_X:TM\to\R$ associated to $X$ is the set of chain-recurrent points of the flow of $X$ on $M$.
Moreover, the constants are the only weak KAM solutions for $L_X$ if and only if
every point of $M$ is chain-recurrent under the flow of  $X$.
\end{theorem}

The outline of the paper is the following: Sections 2 and 3 are devoted to preparatory results. Section 4 is devoted to the proofs of Theorems \ref{THM1}, \ref{THM2} and \ref{THM3}. Sections 5 and 6 present applications in dynamics.

\section{Preliminary results}
Throughout this section, $M$ is assumed to be a complete Riemannian manifold. As before, $H:T^*M \rightarrow \R$ is an Hamiltonian of class at least {\rm C}$^2$ satisfying the three usual conditions (H1)-(H3), and $L$ is the Tonelli Lagrangian which is associated to it by Fenchel's duality. 

\subsection{Some facts about the Aubry set}\label{DefAubry}
We recall the results of Mather on the Aubry set, and also an important complement due to Dias Carneiro.

The following results are due to Mather, see \cite{mather91, mather93} for the proof in the compact case.
\begin{theorem}[Mather]\label{ThMather1}{There exists a closed subset $\tilde \Aub\subset TM$ such that:
\begin{itemize}
\item[(1)] The set $\tilde \Aub$ is invariant under the Euler-Lagrange flow.
\item[(2)] The projection $\pi:TM\to M$ is injective on $\tilde \Aub$. Moreover,
we have $\pi(\tilde \Aub)=\Aub$, and the inverse map $(\pi|\tilde \Aub)^{-1}:\Aub\to
\tilde \Aub$ is locally  Lipschitz.
\item[(3)] Let $(x,v)$ be in $\tilde \Aub$, and call $\gamma_{(x,v)}$ the curve which is the projection   of the  orbit $\phi^L_t(x,v)$ of the Euler-Lagrange flow through $(x,v)$
$$\gamma_{(x,v)}(t)=\pi\phi_t(x,v).$$
This curve is entirely contained in $\Aub$, and it is an $L$-minimizer. Moreover, we have
$$\forall t,t'\in \R, \quad\delta_M(\gamma(t),\gamma(t'))=0,$$
therefore the whole curve $\gamma_{(x,v)}$ projects to the same point as $x$ in the Mather quotient.
\item[(4)] If $x\in  \Aub$ and $\gamma_n:[0,t_n]\to M$ is a sequence of $L$-minimizer such that $t_n\to+\infty, \gamma_n(0)=\gamma_n(t_n)=x$, and 
$\int_0^{t_n}L(\gamma_n(s),\dot\gamma_n(s))\,ds+c(H)t_n\to 0$, then both sequences 
$\dot\gamma_n(0),\dot\gamma_n(t_n)$ converge in $T_xM$ to the unique $v\in T_xM$ such that $(x,v)\in \tilde\Aub$.
\end{itemize}}
\end{theorem}
The following theorem of Dias Carneiro \cite{carneiro} is a nice complement to the Theorem above:
\begin{theorem}\label{ThCarneiro}{For every $(x,v)\in\tilde \Aub$, we have
$$H\left(x,\frac{\partial L}{\partial v}(x,v)\right)=c(H).$$
}
\end{theorem}
We end this subsection by the following important estimation of the Mather semi-distance  (due to  Mather),
see  \cite[page 1375]{mather93}.
\begin{proposition}\label{HolderOrdre2}{ For every compact subset $K\subset M$, we can find a finite constant $C_K$, such that
$$\forall x\in \mathcal{A}\cap K,\, \forall y\in K, \quad\delta_M(x,y)\leq C_Kd(x,y)^2,$$
where $d$ is the Riemannian distance on $M$.}
\end{proposition}
Note that one can prove directly this proposition from the fact that $h$ is locally semi-concave on $M\times M$, using that $\delta_M\geq 0$, together with the fact that $\delta_M(x,x)=0$ for every $x\in \mathcal{A}$.

\subsection{Aubry set and Hamilton-Jacobi equation}
In this section we recast the above results in terms of viscosity solutions of the Hamilton-Jacobi as is done in \cite{fathibook,fs04,fm}.

We first  recall the notion of domination. If $c\in\R$, a function $u:M\to\R$ is said to 
be dominated by $L+c$ (which we denote by $u\prec L+c$), if for every continuous 
piecewise {\rm C}$^1$ curve $\gamma:[a,b]\to M,a<b$, we have
\begin{equation*}
u(\gamma(b))-u(\gamma(a))\leq 
\int_a^bL(\gamma(s),\dot\gamma(s))\,ds+c(b-a).\tag{DOM}
\end{equation*}

In fact this is simply a different way to define the notion of viscosity solution for $H$.
More precisely we have, see \cite{fathibook} or \cite[Proposition 5.1, page 12]{fm}:
\begin{theorem}{A $u:M\to \R$ is dominated by $L+c$ if and only if
it is a viscosity subsolution of the Hamilton-Jacobi equation $H(x,d_xu)=c$.
Moreover, we have $u\prec L+c$ if and only if $u$ is Lipschitz and $H(x,d_xu)\leq c$ almost everywhere.}
\end{theorem}
Note that Rademacher's Theorem states that every Lipschitz function is differentiable almost everywhere. 
For the proof that dominated functions are Lipschitz see \ref{globalLip}.
It is not difficult to see that a function $u:M\to\R$ is dominated by $L+c$ if and only if
$$\forall t>0,\,\forall x,y\in M,\quad u(y)-u(x)\leq h_t(x,y)+ct.$$
With these notations, we observe that a function $u$ is a critical subsolution if and only 
$u\prec L+c(H)$.

We now give the definition of calibrated curves. If $u:M\to\R$ and $c\in \R$,
we say that the curve $\gamma:[a,b]\to M$ is $(u,L,c)$-calibrated 
if  we have the equality
$$u(\gamma(b))-u(\gamma(a))=
\int_a^bL(\gamma(s),\dot\gamma(s))\,ds+c(b-a).$$
If $\gamma$ is a curve defined on the not necessarily compact interval $I$, we will say that  $\gamma$ is $(u,L,c)$-calibrated if its restriction to any compact subinterval of $I$ is $(u,L,c)$-calibrated.

In fact, this condition of calibration is useful only when $u\prec L+c$.
In this case $\gamma$ is an $L$-minimizer. Moreover, if $[a',b']$ is a subinterval of $[a,b]$, then the restriction $\gamma|[a',b']$ is also $(u,L,c)$-calibrated. 

Like in \cite{fathibook}, if $u:M\rightarrow  \R$ is a critical  subsolution, we denote by
$\tilde{\cal I}(u)$ the subset of $TM$ defined as
$$
\tilde{\cal I}(u)= \{(x,v)\in TM \mid \text{$\gamma_{(x,v)}$ is $(u,L,c(H))$-calibrated}\},$$
where $\gamma_{(x,v)}$ is the curve (already introduced in Theorem \ref{ThMather1}) defined on  $\R$ by
$$\gamma_{(x,v)}(t)=\pi \phi^L_t(x,v).$$

The following properties of $\tilde{\cal I}(u)$ are shown in \cite{fathibook}:
\begin{theorem}\label{propertiesI}
{The set $\tilde{\cal I}(u) $ is invariant under the Euler-Lagrange flow $\phi_t^L$. If $(x,v)\in \tilde{\cal I}(u)$, then $d_xu$ exists, and we have
$$
d_xu=\frac{\partial L}{\partial v}(x,v) \quad \text{ and } \quad H(x,d_xu)=c(H).
$$
It follows that the restriction $\pi|_{\tilde{\cal I}(u)}$ of the projection is injective; therefore, if we set ${\cal I}(u)=\pi(\tilde{\cal I}(u))$, then $\tilde{\cal I}(u)$ is a continuous graph over ${\cal I}(u)$.

Moreover, the map $x\mapsto d_xu$ is locally Lipschitz on ${\cal I}(u)$.

Since the inverse of the restriction
$\pi|\tilde{\cal I}(u)$ is given by $x\mapsto \mathcal{L}^{-1}(x,d_xu)$, and the Legendre transform  $\mathcal{L}$ is {\rm C}\/$^1$, it follows that the inverse of $\pi|\tilde{\cal I}(u)$
is also locally Lipschitz on ${\cal I}$.}
\end{theorem}

Using the sets $\tilde{\cal I}(u)$, one can give the following characterization of the Aubry set  and its projection:
\begin{theorem}\label{DescriptionAubry}
{The Aubry set $\tilde {\cal A}$ is given by 
$$
\tilde {\cal A}=\bigcap_{u\in {\cal {SS}}}\tilde{\cal I}(u),
$$
where $\mathcal{SS}$ is the set of critical viscosity subsolutions.
The projected Aubry set $\Aub$, which is simply the image $\pi(\tilde {\cal A})$, is also 
$$ 
{\cal A}=\bigcap_{u\in {\cal {SS}}}{\cal I}(u).
$$}
\end{theorem}
Note that the fact that the Aubry set is a locally Lipschitz graph (i.e. part (2) of Theorem \ref{ThMather1}) follows from the above results, since $\tilde \Aub\subset \tilde{\cal I}(u)$, for any critical subsolution $u$. Moreover,  Theorem \ref{ThCarneiro} also follows from the results above.

\subsection{Mather semi-distance and critical subsolutions}\label{MatherDistAndSubsol}
As it was observed by the first author to generalize Mather's examples \cite{mather04}, see the announcement \cite{fathiOberwolfach}, a representation formula for $\delta_M$ in term of {\rm C}$^1$ critical subsolutions is extremely useful. This has also been used more recently  by Sorrentino \cite{sorrentino06}.

To explain this representation formula, like in Theorem \ref{DescriptionAubry}, we call $\mathcal{SS}$ the set of critical viscosity subsolutions and by ${\cal S}_-$ the set of critical viscosity (or weak KAM) solutions. Hence ${\cal S}_- \subset \mathcal{SS}$. If $u:M \rightarrow \R$ is a critical viscosity subsolution, we recall that 
$$
\forall x,y \in M,\quad
u(y) -u(x) \leq h(x,y).
$$ 
In \cite{fs04}, Fathi and Siconolfi proved that for every critical viscosity subsolution $u:M \rightarrow \R$, there exists a {\rm C}$^1$ critical subsolution whose restriction to the projected Aubry set is equal to $u$. Recently Patrick Bernard \cite{bernard} has even shown that $u$ can be assumed {\rm C}$^{1,1}$, i.e. differentiable everywhere with (locally) Lipschitz derivative, see also Appendix \ref{BernardNonCompact} below. In the sequel, we denote by $\mathcal{SS}^1$ 
(resp.\ $\mathcal{SS}^{1,1}$) the set of {\rm C}$^1$ (resp.\ {\rm C}$^{1,1}$) critical subsolutions. 
The representation formula is given by the following lemma:
\begin{lemma}\label{Fund}{\sl
For every $x,y\in \Aub$,
\begin{align*}
\delta_M (x,y) & = \max_{u_1, u_2\in \mathcal{S}_-} \left\{ (u_1-u_2)(y) - (u_1-u_2)(x) \right\} \\
& = \max_{u_1, u_2\in \mathcal{SS}} \left\{ (u_1-u_2)(y) - (u_1-u_2)(x) \right\} \\
& = \max_{u_1, u_2\in \mathcal{SS}^1} \left\{ (u_1-u_2)(y) - (u_1-u_2)(x) \right\} \\
& = \max_{u_1, u_2\in \mathcal{SS}^{1,1}} \left\{ (u_1-u_2)(y) - (u_1-u_2)(x) \right\}. 
\end{align*}}
\end{lemma}

\begin{proof}
Let $x,y \in \mathcal{A}$ be fixed. First, we notice that if $u_1, u_2$ are two critical viscosity subsolutions, then we have 
\begin{align*}
 (u_1-u_2)(y) - (u_1-u_2)(x) & = \left(u_1(y) - u_1(x)\right) + \left( u_2(x) -u_2(y) \right) \\
& \leq h(x,y) + h(y,x) =\delta_M(x,y).
\end{align*}
On the other hand, if we define $u_1,u_2:M \rightarrow \R$ by $u_1(z)=h(x,z)$ and $u_2(z)=h(y,z)$ for any $z\in M$, by the properties of $h$ the functions $u_1, u_2$ are both critical viscosity solutions. Moreover
\begin{align*}
(u_1-u_2)(y) - (u_1-u_2)(x) & = \left( h(x,y) - h(y,y) \right) - \left( h(x,x) -h(y,x) \right) \\
& = h(x,y) + h(y,x) =\delta_M (x,y),
\end{align*}
since $h(x,x)=h(y,y)=0$. Thus we obtain easily the first and the second equality. The last inequalities is an immediate consequence of the work of Fathi and Siconolfi and that of Bernard recalled above.
\end{proof}

\subsection{Norton's generalization of Morse Vanishing Lemma}
We will need in a crucial way Norton's elegant generalization of Morse Vanishing Lemma, see  \cite{morse39,norton86}. This result, like Ferry's Lemma (see Lemma \ref{lemNEW2}) are the two basic pieces that allow to prove generalizations of the Morse-Sard Theorem (see for example the work of Bates).

\begin{lemma}[The Generalized Morse Vanishing Lemma]
\label{lemmaMorseVanish} Suppose $M$ is an $n$-dimensional (separable) manifold endowed with a distance $d$  coming from a Riemannian metric.
Let $k \in \N$ and $\a \in [0,1]$. Then for any subset $A \subset M$, we can find a countable family  $B_i,i\in \N$ of {\rm C}$^1$-embedded compact disks in $M$ of dimension $\leq n$ and a countable decomposition of $A=\cup_{i \in \N} A_i$, with $A_i \subset B_i$,  for every $i\in \N$,  such that
every $f \in {\rm C}^{k,\a}(M,\R)$ vanishing on $A$ satisfies, for each $i \in \N$,
\begin{equation}\label{eqmorselemma}
\forall y \in A_i, \, x \in B_i, \quad 
|f(x)-f(y)| \leq M_i d(x,y)^{k+\a} 
\end{equation}
for a certain constant $M_i$ (depending on $f$).
\end{lemma}

Let us make some comments. In his statement of the Lemma above (see \cite{norton86}), Norton  distinguishes a countable $A_0$ in his decomposition. In fact, in the 
statement we give this corresponds to the (countable numbers of) 
disks in the family $B_i$ 
where the dimension of the disk $B_i$ is $0$, in which case $A_i$ is also a point.
Therefore there is no need to distinguish this countable subset when formulating 
the Generalized Morse Vanishing Lemma. The second comment is that we have 
stated this Generalized Morse Vanishing Lemma \ref{lemmaMorseVanish} directly for (separable) manifolds. This is a routine generalization of the case $M=\R^n$ which is done by Norton in \cite{norton86} (see for example the way we deduce Lemma \ref{lemNEW2} from  Lemma \ref{lemNEW1}).

\section{Proofs of Theorems \ref{THM1}, \ref{THM2}, \ref{THM3}, and \ref{THMSURF}} 

\subsection{Proof of Theorem \ref{THM1}}

Let us first assume that $\dim M =1,2$. The proof is the same as Mather's proof of total
disconnectedness given in \cite{mather03}. It also uses Proposition \ref{HolderOrdre2}, 
but instead of using the results of Mather contained in \cite{mather02}, it uses the stronger Lemma \ref{lemNEW2} due to Ferry and proved in Appendix A below.

We cover $M$ by an increasing countable union $K_n$ of compact subsets. For a given $n$, by Proposition \ref{HolderOrdre2} we can find  a finite constant $C_n$ such that
$$\forall x,y\in\mathcal{A}\cap K_n, \quad\delta_M(x,y)\leq C_nd(x,y)^2.$$
Since $\dim M\leq 2$ by  Lemma \ref{lemNEW2} we obtain that  $(\Aub\cap K_n,\delta_M)$  has vanishing one-dimensional Hausdorff measure.  Since $\Aub$ is the countable union of the  $\Aub\cap K_n$, we also conclude that $(\Aub,\delta_M)$  has vanishing one-dimensional Hausdorff measure.

Let us now assume that $\dim M = 3$. The fact that $(\mathcal{A}^0_M,\delta_M)$ has vanishing one-dimensional Hausdorff measure will follow from Theorem \ref{THM2}. So, it suffices to prove that the semi-metric space $(\mathcal{A}\setminus \mathcal{A}^0,\delta_M)$ has vanishing one-dimensional Hausdorff measure.  

Consider for every $x\in \Aub$ the unique vector $v_x\in T_xM$ such that 
$(x,v_x)\in \tilde\Aub$. Call $\gamma_x$ the curve defined by 
$\gamma_x(t)=\pi\phi^L_t(x,v_x)$. 
Since $\tilde\Aub$ is invariant by $\phi^L_t$, the projected Aubry set is 
laminated by the curves $\gamma_x,x\in\Aub$. Let us define 
$\mathcal{A}'= \mathcal{A}\setminus \mathcal{A}^0$. Since, by  Proposition \ref{caracAub0},
any point of the form $(z,0)\in \tilde\Aub$ is fixed under $\phi^L_t$, and 
$\dot\gamma_x(0)=v_x$, we have 
$\gamma_x(t)\in \Aub'$ for all $x\in \Aub'$ and all $t\in \R$. Moreover, 
the family $\gamma_x,x\in \Aub'$, is a genuine 1-dimensional Lipschitz 
lamination on  
$\mathcal{A}'= \mathcal{A}\setminus \mathcal{A}^0$. 
For each $x\in \mathcal{A}'$, we can find a small {\rm C}$^\infty$ 2-dimensional  
submanifold $S_x$ of $M$ such that $S_x$ is transversal to $\gamma_x$.
By transversality and continuity, the union $U_x$ of the curves $\gamma_y,y\in \mathcal{A}'$ such that $\gamma_y\cap S_x\neq\emptyset$ is a neighborhood of $x$ in $\Aub'$ (for the topology induced by the manifold topology).
Therefore since $M$ is metric separable, we can find a countable subfamily $(S_{x_i})_{i\in \N}$ such that $\gamma_y\cap(\cup_{i\in \N}S_{x_i})\neq \emptyset$ for every $y\in \Aub'$. 
By part (3) of Theorem \ref{ThMather1} above, for every $z\in \Aub$, and every $ t,t'\in\R$, we have  
$$\delta_M(\gamma_z(t),\gamma_z(t'))=0.$$ 
It follows that the countable union of the images of $S_{x_i}\cap\Aub$ in $\Aub_M$
covers the image of $\Aub'$ in $\Aub_M$. Therefore by the countable additivity of 
the Hausdorff measure, we have to show that $(S_{x_i}\cap \Aub,\delta_M)$ has 
 1-dimensional Hausdorff measure equal to $0$. Since $S_{x_i}$ is 2-dimensional, this follows from Proposition \ref{HolderOrdre2} and Lemma \ref{lemNEW2} like above.

\subsection{Proof of Theorem \ref{THM2}}
Before giving the proof we need a better understanding of the sets $\tilde\Aub^0$ and $\Aub^0$. 
\begin{lemma}\label{propTildeH}{The function $\tilde H:M\to M$ defined by 
$$\tilde H(x)=\inf\{H(x,p)\mid p\in T_x^*M\}$$
satisfies the following properties:
\begin{itemize}
\item[{\rm(i)}] For every $x\in M$, we have $\tilde H(x)\leq c(H)$.
\item[{\rm(ii)}] We have $H(x,p)=\tilde H(x)$ if and only if $p=\partial L/\partial v(x,0)$. 
\item[{\rm(iii)}] For every $x\in M$, we have 
$$\tilde H(x)=H\left(x,\frac{\partial L}{\partial v}(x,0)\right)=-L(x,0).$$
Therefore $\tilde H$ is as smooth as $x\mapsto L(x,0)$.
\item[{\rm(iv)}] The point $x$ is a critical point of $\tilde H$ (or of $x\mapsto L(x,0)$) if and only
the point $(x, \partial L/\partial v(x,0))$ is a critical point of $H$.
In particular, the point $(x, \partial L/\partial v(x,0))$ is a critical point of $H$ for every
$x$ such that $\tilde H(x)=c(H)$.
\end{itemize}}
\end{lemma}
\begin{proof}
Since there exists a  {\rm C}$^1$ critical subsolution  $u:M\to\R$ which satisfies
$$\forall x\in M,\quad H(x,d_xu)\leq c(H),$$
we must have
$$\tilde H(x)=\inf\{H(x,p)\mid p\in T_x^*M\}\leq c(H).$$
By strict convexity the infimum $\tilde H(x)$ is attained at the unique point $\tilde p(x)\in T_x^*M$. which satisfies
\begin{equation*}
\frac{\partial H}{\partial p}(x,\tilde p(x))=0. \tag{$*$}
\end{equation*}
Since $(x,p)\mapsto (x, \partial H/\partial p(x,p))$ is the inverse of the Legendre transform
$(x,v)\mapsto\partial L/\partial v(x,v)$, we obtain
$$\tilde p(x)=\frac{\partial L}{\partial v}(x,0),$$
and therefore by the Fenchel equality
$$\tilde H(x)=H(x,\tilde p(x))=\frac{\partial L}{\partial v}(x,0)-L(x,0)=-L(x,0).$$
To prove the last part (iv), we first observe that 
$$\frac{\partial H}{\partial p}\left(x,  \frac{\partial L}{\partial v} (x,0)\right)=0.$$
Then we differentiate (in a coordinate chart) the equality obtained in (ii) to obtain
\begin{align*}
d_x\tilde H &=\frac{\partial H}{\partial x}\left(x,  \frac{\partial L}{\partial v}(x,0)\right) +
\frac{\partial H}{\partial p}\left(x, \frac{\partial L}{\partial v}(x,0)\right)\circ \frac{\partial^2 L}{\partial v^2}(x,0)\\
&=\frac{\partial H}{\partial x}\left(x,  \frac{\partial L}{\partial v}(x,0)\right).
\end{align*}
Therefore, by the two previous equations, the first  part of (iv) follows.
The last part of (iv) is a consequence (i), which implies that each $x$ satisfying $\tilde H(x)=c(H)$ is a global maximum of $\tilde H$.
\end{proof}
We can now give a characterization of the stationary Aubry set $\tilde \Aub^0$.
\begin{proposition}\label{caracAub0}{The set $\tilde\Aub^0$ of points in $\tilde\Aub$ which are fixed for the Euler-Lagrange flow $\phi_t^L$ is exactly the intersection of $\tilde\Aub$ with the zero section in $TM$, i.e.
$$\tilde\Aub^0=\tilde\Aub\cap\{(x,0)\mid x\in M\}.$$
Its projection $\Aub^0=\pi(\tilde\Aub^0)$ on $M$ is precisely the set of point $x$ in $M$
at which $\tilde H$ takes the value $c(H)$, i.e.
$$\Aub^0=\{x\in M\mid \tilde H(x)=0\}.$$}
\end{proposition}
\begin{proof}Let $(x,v)$ be in $\tilde\Aub^0$. Since the Euler-Lagrange flow $\phi_t^L$ is conjugated to the Hamiltonian flow $\phi^H_t$ of $H$ by the Legendre transform $\mathcal{L}$, we obtain that $(x,\partial L/\partial v(x,v))$ is fixed under $\phi^H_t$, and therefore $(x,\partial L/\partial v(x,v))$ is a critical point of $H$. In particular, we have
$$\frac{\partial H}{\partial p}\left(x, \frac{\partial L}{\partial v}(x,v)\right)=0.$$
Since $(x,p)\mapsto (x,\partial H/\partial p(x,p))$ is the inverse of the Legendre transform, we conclude that $v=0$, yielding the proof of the inclusion
$\tilde\Aub^0\subset\tilde\Aub\cap\{(x,0)\mid x\in M\}$.

Suppose now that that $(x,0)$ is in $\tilde\Aub$. 
By Theorem \ref{ThCarneiro},
the Legendre transform of the Aubry set is contained in the set where $H$ is equal to $c(H)$, i.e.
$$H\left(x,\frac{\partial L}{\partial v}(v,0)\right)=c(H).$$
We obtain by Lemma \ref{propTildeH} that $x$ is a critical point of $\tilde H$ and therefore we get that 
$(x,{\partial L}/{\partial v}(v,0))$ is a critical point of $H$. This implies that this points
is invariant under $\phi^H_t$, hence $(x,0)$ is fixed under the Euler-Lagrange flow $\phi^L_t$.
By this we get the equality $\tilde\Aub^0=\tilde\Aub\cap\{(x,0)\mid x\in M\}$.

Note that we have proved that if $(x,0)$ in $\tilde\Aub$ then $\tilde H(x)=c(H)$.
Therefore $\Aub^0$ is contained in the set $\tilde H^{-1}(c(H))$.

It remains to show that any $x$ such that $\tilde H(x)=c(H)$ is in $\Aub^0$.
Suppose that $x$ is such that $\tilde H(x)=c(H)$. Since $\tilde H(x)=-L(x,0)$, we get  
$L(x,0)+c(H)=0$. If we consider now the constant curve $\gamma:]-\infty,+\infty[\to\{x\}$, we see that
$$\int_0^tL(\gamma(s),\dot\gamma(s))\,ds+c(H)t=\int_0^tL(x,0)\,ds+c(H)t=0.$$
Therefore $h_t(x,x)+c(H)t=0$ for every $t\geq 0$. This implies that $x\in \Aub$.
It remains to show that the point $(x,v)\in \tilde\Aub$ above $x$ is necessarily $(x,0)$,
which will imply that $x\in\Aub^0$. 
Note that again by Theorem \ref{ThCarneiro}
we have 
$$H\left(x,\frac{\partial L}{\partial v}(x,v)\right)=c(H).$$
But $\inf_{p\in T^*_xM}H(x,p)=\tilde H(x)=c(H)$, and that this infimum is only attained at $p={\partial L}/{\partial v}(x,0)$. This implies that ${\partial L}/{\partial v}(x,v)={\partial L}/{\partial v}(x,0)$. The invertibility of the Legendre transform yields $v=0$.
\end{proof}

We now start the proof of Theorem \ref{THM2}. Replacing $L$ by $L+c(H)$, we can assume, without loss of generality, that $c(H)=0$.
Notice now that, for every compact subset  $K\subset M$, there exists
$\alpha_K \geq 0$ such that 
\begin{equation*}
x\in K, \, H(x,p) \leq 0 \Longrightarrow \lVert p-\tilde{p}(x)\rVert_x \leq \alpha \sqrt{-\tilde{H}(x)}.\tag{$**$}
\end{equation*}
In fact, since $\partial ^2H/\partial p^2(x,p)$ is positive definite everywhere, and $S(K)=\{(x,p)\in T^*M\mid x\in K,\, H(x,p)\leq 0\}$ is compact, by Taylor's formula (in integral form), we can find $\beta_K >0$ such that 
$$\forall (x,p),(x,p')\in S(K),\quad H(x,p)\geq H(x,p')+\frac{\partial H}{\partial p}(x,p')(p-p')+\beta_K
\lVert p-p'\rVert^2_x.$$
Using the equalities ($*$) and $\tilde H(x)=H(x,\tilde p(x))$, and that $H(x,p)\leq 0$ on $S(K)$,
the inequality above yields
$$\forall (x,p)\in S(K),\quad 0\geq \tilde H(x)+\beta_K\lVert p-\tilde p(x)\rVert^2_x.$$
This yields ($**$) with $\alpha_K=\sqrt{\beta_K}$.
If $u:M\to \R$ is a {\rm C}$^1$ critical subsolution, we know that $H(x,d_xu)\leq 0$ for every $x\in M$, therefore we obtain
$$\forall x\in M,\quad \lVert d_xu-\tilde p(x)\rVert_x\leq \alpha_K\sqrt{-\tilde H(x)}.$$
It follows that for every pair $u_1,u_2$ of critical subsolutions,  we have
\begin{equation*}
\forall x\in M,\quad \lVert d_x(u_2-u_1)\rVert_x\leq 2\alpha_K\sqrt{-\tilde H(x)}. \tag{$***$}
\end{equation*}

We now use Lemma \ref{lemmaMorseVanish} for {\rm C}$^{k,1}$ functions to decompose $\mathcal{A}^0$ as
$$
\mathcal{A}^0 =\cup_{i \in \N} A_i,
$$
with each $A_i\subset B_i$, where $B_i\subset M$ is a {\rm C}$^1$ embedded compact disk of dimension $\leq \dim M$. Since $\tilde{H}$ is a {\rm C}$^{k,1}$ function vanishing on $\mathcal{A}^0$,
by (\ref{eqmorselemma}) we know that we can find for each $i\in \N$ a finite constant $M_i$ such that 
$$
\forall x \in A_i, \,\forall y \in B_i,\quad -\tilde{H}(y) = \left| \tilde{H}(x)-\tilde{H}(y) \right| \leq M_i d(x,y)^{k+1}.
$$
Since $B_i$ is compact, we can combine this last inequality with ($***$) above to obtain
for every pair of critical subsolutions $u_1,u_2$, and every $i\in\N$
$$
 \forall x \in A_i, \,\forall y \in B_i, \quad \lVert d_y(u_2-u_1)\rVert_y\leq 2\alpha_{B_i}\sqrt{M_i} d(x,y)^{(k+1)/2}.
$$
We know that $B_i$ is {\rm C}$^1$ diffeomorphic to the unit ball $\BB^{n_i}$, with $n_i\in \{0,\dots,\dim M\}$. To avoid heavy notation we will identify in the sequel of the proof $B_i$ with  $\BB^{n_i}$. Since this identification is {\rm C}$^1$, we can replace in the inequality above the Riemannian norm by the Euclidean norm $\lVert \cdot\rVert_{\rm{euc}}$ on $\R^{n_i}$, to obtain the following inequality
$$
 \forall x \in A_i, \,\forall y \in B_i\approx\BB^{n_i}, \quad \lVert d_y(u_2-u_1)\rVert_{\rm euc} 
 \leq  C_i\lVert y-x\rVert_{\rm euc}^{(k+1)/2}.
$$
for some suitable finite constant depending on $i$. If we integrate this inequality along the segment from $x$ to $y$ in  $\BB^{n_i}\approx B_i$, we obtain
$$
 \forall x \in A_i, \,\forall y \in B_i\approx\BB^{n_i},\quad |(u_1-u_2)(y)-(u_1-u_2)(x)| \leq  
 C_i \lVert y-x\rVert_{\rm euc}^{\frac{k+1}{2}+1}.
$$
By Lemma \ref{Fund} we deduce that 
$$
\forall x,y \in A_i,\quad \delta_M(x,y) \leq \tilde C_i \lVert y-x\rVert_{\rm euc}^{\frac{k+1}{2}+1}.
$$
Since $A_i\subset B_i\approx\BB^{n_i}\subset \R^{n_i}$, and obviously $1+\frac{k+1}{2}>1$, we conclude from Lemma \ref{lemNEW1}
that the Hausdorff measure $\Haus {n_i/(1+\frac{k+1}{2})}(A_i,\delta_M)$ 
is equal to $0$. 
Therefore, since $n_i\leq \dim M$ and ${\Aub}^0$ is the countable union of the $A_i$, 
we conclude that 
$$\Haus {2\dim M/(k+3)}(\Aub^0,\delta_M)=0.$$
In particular, if $k+3\geq 2\dim M$, that is  $k \geq 2n-3$, the one-dimensional Hausdorff dimension of $(\mathcal{A}_M^0,\delta_M)$ vanishes.

\subsection{Proof of Theorem \ref{THM3}}
We will give a proof of  Theorem \ref{THM3} that does not use conservation of energy (complicating a little bit some of the steps). It will use instead the completeness of the Euler-Lagrange flow, which is automatic for Tonelli Lagrangians independent of time, 
see \cite[Corollary 2.2, page 6]{fm}.
It can therefore be readily adapted to the case where $L$ depends on time, is $1$-periodic in time and has a complete Euler-Lagrange flow like in the work of Mather
\cite{mather91, mather93}.

In a flow the period function on the periodic non-fixed point is not necessarily continuous.
Therefore when we pick a local Poincar\'e section for a closed orbit the nearby periodic points of the flow do not give rise to fixed point of the Poincar\'e return map. This will cause us some minor difficulties in the proof of Theorem \ref{THM3}. We will use the following general lemma to easily get around  these problems.
\begin{proposition}\label{lemmepointsperiodiques}{ Let $X$ be a metric space and
 $(\phi_t)_{t\in\R}$ be a continuous flow on $X$. Call $\operatorname{Fix}(\phi_t)$
the set of fixed points of the flow $(\phi_t)_{t\in\R}$, and $\operatorname{Per}(\phi_t)$
the set of periodic non-fixed points of $(\phi_t)_{t\in\R}$. We define the function $T:\operatorname{Per}(\phi_t)\to]0,\infty[$ by $T(x)$ is the smallest period $>0$ of the point $x\in\operatorname{Per}(\phi_t)$.

We can write $\operatorname{Per}(\phi_t)$ as a countable union $\operatorname{Per}(\phi_t)=\cup_{n\in \N}C_n$ where each $C_n$ is a closed subset on which the period map $T$ is continuous.}
\end{proposition}
\begin{proof} For $t\in \R$, call $F_t$ the set of fixed points of the map $\phi_t$. Using the continuity of $(\phi_t)_{t\in\R}$ on the product on $\R\times X$, it is not difficult to see
that  $\cup_{t\in [a,b]}F_t$ is a closed subset of $X$ for every compact subinterval $[a,b]$ contains in $\R$. For $n\in \Z$ we set 
$$F^n=\cup_{2^n\leq t\leq 2^{n+1}}F_t=\{x\in X\mid \exists t\in [2^n,2^{n+1}]\text{\ with\ }
\phi_t(x)=x.$$
Note that $F^m$ is closed. Moreover 
since $\phi_t(x)=x$ with $2^{m-1}\leq t\leq 2^{m}$ implies 
$\phi_{2t}(x)=\phi_t\circ\phi_t(x)=\phi_t(x)=x$ and $2^{m}\leq 2t\leq 2^{m+1}$, we get
$F^{m-1}\subset F^m$ for every $m\in \Z$,.
Therefore we have $F^n\setminus F^{n-1}=F^n\setminus \cup_{i\leq n-1}F^i$ 
is the set of periodic non-fixed points with $2^{n}<T(x)\leq 2^{n+1}$. 
In particular $\operatorname{Per}(\phi_t)=\cup_{n\in\Z}F^n\setminus F^{n-1}$. 
Note also that if $x\in F^n\setminus F^{n-1}$, and $t\in ]0,2^{n+1}[$ are such that
$\phi_t(x)=x$ then necessarily $t=T(x)$. In fact, we have $t/T(x)\in \N^*$, but 
$t/T(x)\leq2^{n+1} /T(x)<2$, hence $t/T(x)=1$.

We now show that the period map $T$ is continuous on $F^n\setminus F^{n-1}$.
For this we have to show that for a sequence $x_\ell\in F^n\setminus F^{n-1}$ 
which converges to $x_\infty\in F^n\setminus F^{n-1}$, we necessarily have 
$T(x_\ell)\to T(x_\infty)$, when $\ell\to \infty$. Since $T(x_\ell)\in[2^n,2^{n+1}]$ 
which is compact, it suffices to show that any accumulation point $T$ of $T(x_\ell)$
satisfies $T= T(x_\infty)$. Pick up an increasing subsequence $\ell_k\nearrow\infty$
such that $T(x_{ \ell_k})\to T$ when $k\to \infty$. By continuity $T\in [2^n,2^{n+1}]$ and
$\phi_T(x_\infty)=x_\infty$. Since $x_\infty\in F^n\setminus F^{n-1}$, by what we have shown above we have $T=T(x_\infty)$.

Since $\operatorname{Per}(\phi_t)$ is the countable union $\cup_{n\in\Z}F^n\setminus F^{n-1}$, to finish the proof of the lemma it remains to show that each $F^n\setminus F^{n-1}$ is itself a countable union of closed subsets of $X$. This is obvious because
$F^n\setminus F^{n-1}=F^n\cap (X\setminus F^{n-1})$ is the intersection of a closed and an open subset in the metric $X$, but an open subset in a metric space is itself a countable union of closed sets.
\end{proof}

We will also need the following proposition which relates the size of the derivative of a {\rm C}$^{1,1}$ critical sub-solution at a point to minimal actions of loops at that point.
We will need to use Lipschitz functions from a compact subset of $M$ to a compact subset of $TM$. We therefore need distances on $M$ and $TM$. On $M$ we have already a distance coming from the Riemannian metric. Since all distances, obtained from Riemannian metrics, are Lipschitz equivalent on compact subsets, the precise distance we use on $TM$ is not important. We therefore just assume that we have chosen some Riemannian metric on $TM$ (not necessarily related to the one on $M$),
and we will use the distance on $TM$ coming from this Riemannian metric.

\begin{proposition}\label{propositionderivativeaction}{ Suppose that $K$ is a given compact set, and $t_0,t'_0\in \R$ satisfy
$0<t_0\leq t'_0$. We can find a compact set $K'$ such that, for any finite number $\ell $,
we can find a finite number $C$ such that any critical {\rm C}$^1$ subsolution
$u:M\to \R$,  such that $x\mapsto (x,d_xu)$ is Lipschitz on $K'$ with Lipschitz constant $\leq \ell$, satisfies
$$\forall x\in K,\,\forall t\in [t_0,t'_0],\quad [c(H)-H(x,d_xu)]^2\leq C[h_{t}(x,x)+c(H)t].$$
Moreover, for every such $\ell$, we can find a constant $C'$ such that any pair of critical {\rm C}$^1$ subsolutions
$u_1,u_2:M\to \R$, such that both maps $x\mapsto (x,d_xu_i), i=1,2,$ are Lipschitz on $K'$ with Lipschitz constant $\leq \ell$, satisfies
$$\forall x\in K,\,\forall t\in [t_0,t'_0], \quad\lVert d_xu_2-d_xu_1\rVert_x^4
\leq C' [h_{t}(x,x)+c(H)t].$$ }
\end{proposition}
When $M$ is compact, we can take $t'_0=+\infty$, and the above Proposition becomes:
\begin{proposition}\label{propositioncompactderivativeaction}{ Suppose the manifold $M$ is compact, and that $t_0>0$ is given.
For any finite number $\ell $,
we can find a finite number $C$ such that any critical {\rm C}$^1$ subsolution
$u:M\to \R$,  such that $x\mapsto (x,d_xu)$ is Lipschitz on $M$ with Lipschitz constant $\leq \ell$, satisfies
$$\forall x\in M,\,\forall t\geq t_0, \quad[c(H)-H(x,d_xu)]^2\leq C[h_{t}(x,x)+c(H)t].$$
Moreover, for every such $\ell$, we can find a constant $C'$ such that any pair of critical {\rm C}$^1$ subsolutions
$u_1,u_2:M\to \R$, such that both maps $x\mapsto (x,d_xu_i), i=1,2,$ are Lipschitz on $M$ with Lipschitz constant $\leq \ell$, satisfies
$$\forall x\in M,\,\forall t\geq t_0, \quad\lVert d_xu_2-d_xu_1\rVert_x^4
\leq C' [h_{t}(x,x)+c(H)t].$$ }
\end{proposition}

To prove these propositions, we first need to prove some lemmas.

\begin{lemma}\label{lemmecompaciteminimizer}{Suppose $K$ is a compact subset of $M$ and that $t_0,t'_0\in \R$ satisfy
$0<t_0\leq t'_0$. We can find a compact subset $K'\subset M$ containing $K$ (and depending on $K, t_0,t'_0$) such that any $L$-minimizer $\gamma:[a,b]\to M$ with
$t_0\leq b-a\leq t'_0$, and $\gamma(a),\gamma(b)\in K$ is contained in $K'$.}
\end{lemma}
Of course when $M$ is compact we could take $K'=M$ and the lemma is trivial.
\begin{proof}[Proof of  Lemma \ref{lemmecompaciteminimizer}.] Since $M$ is a complete Riemannian manifold, we can find $g:[a,b]\to M$ a geodesic with $g(a)=\gamma(a)$, $g(b)=\gamma(b)$, and whose length is $d(\gamma(a),\gamma(b))$. Since $g$ is a geodesic, the norm 
$\lVert \dot g(s)\rVert_{g(s)}$ of its speed is a constant that we denote by $C$. Therefore we have
$$d(\gamma(a),\gamma(b))=\operatorname{length}(g)=\int_a^b  \lVert \dot g(s)\rVert_{g(s)}\,ds=C(b-a).$$
This yields that the norm of speed $\lVert \dot g(s)\rVert_{g(s)}=C=d(\gamma(a),\gamma(b))/(b-a)$ is bounded by $\diam (K)/t_0$. 
If we set 
$$A=\sup\{L(x,v)\mid (x,v)\in TM, \lVert v\rVert_x\leq \diam(K)/t_0\},$$
we know that $A$ is finite by the uniform boundedness of $L$ in the fiber.
It follows that we can estimate the action of $g$ by
$$\int_a^bL(g(s),\dot g(s))\,ds\leq A(b-a).$$
Since $\gamma$ is a minimizer with the same endpoints as $g$, we also
get
$$\int_a^bL(\gamma(s),\dot \gamma(s))\,ds\leq A(b-a).$$
By the uniform superlinearity of $L$ in the fibers, we can find a constant $C>-\infty$ such that
$$\forall (x,v)\in TM,\quad C+\lVert v\rVert_x\leq L(x,v).$$
Applying this to $(\gamma(s),\dot \gamma(s))$ and integrating we get
$$C(b-a)+\operatorname{length}(\gamma)\leq \int_a^bL(\gamma(s),\dot \gamma(s))\,ds\leq A(b-a).$$
Therefore
$$\operatorname{length}(\gamma) \leq (A-C)(b-a).$$
Therefore $\gamma$ is contained in the set $K$ defined by
$$K'=\bar V_{(A-C)(b-a)}(K)=\{y\mid \exists x\in K, d(x,y)\leq (A-C)(b-a)\}.$$
Notice that $K'$ is contained in a ball of radius $\bigl(\diam K+(A-C)(b-a)\bigr)$ which is finite,
and balls of finite radius are compact in a complete Riemannian manifold.
Therefore $K'$ is compact.
\end{proof}
\begin{lemma}\label{boundedspeed}{For every compact subset $K'$ of $M$,
and every $t_0>0$, we can find a constant $C=C(t_0,K')$ such that every $L$-minimizer
$\gamma:[a,b]\to M$, with  $b-a\geq t_0$ and $\gamma([a,b])\subset K'$, satisfies
$$\forall s\in [a,b],\quad \lVert \dot\gamma(s)\rVert_{\gamma(s)}\leq C.$$}
\end{lemma}
\begin{proof} Since any $s\in [a,b]$ with $b-a\geq t_0$ is contained in an subinterval
of length exactly $t_0$, and any sub-curve of a minimizer is a minimizer, 
it suffices to prove the lemma under the condition $b-a=t_0$. 
Using the action of a geodesic from $\gamma(a)$ to $\gamma(b)$, and the uniform boundedness of $L$ in he fibers like in the 
proof of the previous Lemma \ref{lemmecompaciteminimizer}, we can find 
a constant $A$ (depending on $\diam (K)$ and $t_0$ but not on $\gamma$) such that
$$\int_a^bL(\gamma(s),\dot \gamma(s))\,ds\leq A(b-a).$$
Therefore, we can find $s_0\in [a,b]$ such that 
$L(\gamma(s_0),\dot \gamma(s_0))\leq A$. By the uniform superlinearity of $L$, the subset 
$${\cal K}=\{(x,v)\in TM \mid x\in K', L(x,v)\leq A\}$$
is compact (and does not depend on $\gamma$). Since $\gamma$ is a minimizer, we have $(\gamma(s),\dot \gamma(s))=\phi_{s-s_0}(\gamma(s_0),\dot \gamma(s_0))$, and $\lvert s-s_0\rvert\leq b-a=t_0$, we conclude that the speed curve of the minimizer $\gamma$ is contained in the  set (independent of $\gamma$)
$${\cal K}'=\bigcup_{\lvert t\rvert\leq t_0}\phi_t^L({\cal K}),$$
which is compact by the continuity of the Euler-Lagrange flow.
\end{proof}
\begin{lemma}\label{lemmemajorationdudefaut}{For every $K$ compact subset of $M$,
every  $t_0>0$ and every $t'_0\in [t_0,+\infty[$ 
(resp.  $t'_0=+\infty$, when $M$ is compact), 
we can find $K'\supset K$ a compact subset of $M$ 
(resp. $K'=M$ when $M$ is compact) and finite constants
$C_0,C_1$ such that:\\
for every {\rm C}$^1$ critical subsolution $u: M\to \R$, if $\omega_{u,K'}:[0,\infty[\to\infty$ is a continuous non-decreasing modulus of continuity of $x\mapsto (x,d_xu)$ on $K'$, then  for every $ x, y\in K$, and every $t\in\R$ with $t_0\leq t\leq t'_0$, we have
\begin{multline*}
\omega_{u,K'}^{-1}\left(\frac{c(H)-H(x,d_xu)}{2C_1}\right)\frac{c(H)-H(x,d_xu)}{2C_0}
\leq h_t(x,y)+c(H)t+u(x)-u(y),
\end{multline*}
where 
$$
\omega_{u,K'}^{-1}(t)
=\left\{
\begin{array}{ll}
\inf \{t' \mid \omega_{u,K'}(t')=t\} & \text{if }t\in \omega_{u,K'}([0,+\infty[),\\
+\infty & \text{otherwise.}
\end{array}
\right.
$$

In particular, if $\omega_{u,K'}$ is the linear function $t\mapsto Ct$, with $C>0$,
then for every $ x, y\in K$, and every $t\in\R$ with $t_0\leq t\leq t'_0$, we have
\begin{equation*}
\frac{[c(H)-H(x,d_xu)]^2}{4CC_0C_1}\leq h_t(x,y)+c(H)t+u(x)-u(y).
\end{equation*}}
\end{lemma}
\begin{proof} We first choose $K'$. If $M$ is compact, we set $K'=M$, and we allow $t'_0=+\infty$. If $M$ is not compact we assume $t'_0<+\infty$. By Lemma
 \ref{lemmecompaciteminimizer}, we can find  a compact subset $K'\supset K$ of $M$ 
 such that every $L$-minimizer $\gamma:[a,b]\to M$ with
$t_0\leq b-a\leq t'_0$, and $\gamma(a),\gamma(b)\in K$ is contained in $K'$.
With this choice of $K'$, we apply Lemma \ref{boundedspeed} to find a finite constant a  finite constant
$C_0$  such that every $L$-minimizer $\gamma:[a,b]\to M$ contained in $K'$, with $b-a\geq 0$, has a speed
bounded in norm by $C_0$.

Therefore we conclude that for every $L$-minimizer $\gamma:[0,t]\to M$, with $t_0\leq  t'_0$ and $\gamma(0),\gamma(t)\in K$, we have $\gamma([0,t])\subset K'$, and $\lVert\dot\gamma(s)\rVert_{\gamma(t)}\leq C_0$ (this is valid both in the compact and non-compact case). In particular, for such a minimizer $\gamma$, we have
$$\forall s,s'\in [0,t], \quad d(\gamma(s),\gamma(s'))\leq C_0\lvert s-s'\rvert.$$
We call $C_1$ a Lipschitz constant of $H$ on the compact subset $\{(x,p)\in T^*M\mid x\in K', H(x,p)\leq c(H)\}$.

Suppose now $u$ is a critical subsolution. Given $x,y\in K$, and $t\in [t_0,t'_0]$, we pick $\gamma:[0,t]\to M$ a minimizer with $\gamma(0)=x$ and $\gamma(t)=y$. 
Therefore we have
\begin{equation*}
h_t(x,y)=\int_0^tL(\gamma(s),\dot\gamma(s))\,ds.
\end{equation*}
Since $\gamma([0,t])\subset K'$ and $H(\gamma(s),d_{\gamma(s)}u)\leq c(H)$, we have 
\begin{multline*}
\forall s,s'\in [0,t], \quad \lvert H(\gamma(s'),d_{\gamma(s')}u)-H(\gamma(s),d_{\gamma(s)}u)\rvert\\
\leq C_1d[(\gamma(s'),d_{\gamma(s')}u),(\gamma(s),d_{\gamma(s)}u)]\\
\leq C_1\omega_{u,K'}(d(\gamma(s'),\gamma(s))\leq C_1\omega_{u,K'}( C_0\lvert s-s'\rvert).  
\tag{$*$}
\end{multline*}
Integrating the Fenchel inequality
$$d_{\gamma(s)}u(\dot\gamma(s))\leq L(\gamma(s),\dot\gamma(s))+H(\gamma(s),d_{\gamma(s)}u),$$
we get
$$u(y)-u(x)\leq h_t(x,y)+\int_0^tH(\gamma(s),d_{\gamma(s)}u)\,ds.$$
Since $H(\gamma(s),d_{\gamma(s)}u)\leq c(H)$, for every $t'\in [0,t]$, we can write 
\begin{align*}
\int_0^tH(\gamma(s),d_{\gamma(s)}u)\,ds&=\int_0^tc(H)+[H(\gamma(s),d_{\gamma(s)}u)-c(H)]\,ds\\
&\leq c(H)t+\int_0^{t'}H(\gamma(s),d_{\gamma(s)}u)-c(H)\,ds\\
&\leq  c(H)t+\int_0^{t'}H(\gamma(0),d_{\gamma(0)}u)-c(H) +
C_1\omega_{u,K'}( C_0 s)\,ds\\
&=  c(H)t+\int_0^{t'}H(x,d_xu)-c(H) +
C_1\omega_{u,K'}( C_0 s)\,ds.
\end{align*}
Therefore from ($*$) above we obtain
$$\forall t'\in [0,t],\quad u(y)-u(x)\leq h_t(x,y)+c(H)t+\int_0^{t'}H(x,d_xu)-c(H) +
C_1\omega_{u,K'}( C_0 s)\,ds,$$
which yields
\begin{equation*}\forall t'\in [0,t], \quad\int_0^{t'}c(H)-H(x,d_xu) -
C_1\omega_{u,K'}( C_0 s)\,ds\leq h_t(x,y)+c(H)t+u(x)-u(y).\tag{$**$}
\end{equation*}
Since $t \geq t_0$ and $c(H)-H(x,d_xu) \leq c(H)-\inf\{ H(x,p) \mid (x,p)\in T^*M\}<+\infty$,
up to choose $C_0$ big enough we can assume $t> \frac1{C_0}\omega_{u,K'}^{-1}(\frac{c(H)-H(x,d_xu)}{2C_1})$.
Then, if we set $t'=\frac1{C_0}\omega_{u,K'}^{-1}(\frac{c(H)-H(x,d_xu)}{2C_1})$,
since $\omega_{u,K'}$ is non decreasing we obtain
$$\forall s\in [0,t'], \quad C_1\omega_{u,K'}(C_0s)\leq \frac{c(H)-H(x,d_xu)}2.$$
Hence
$$\forall s\in [0,t'],\quad c(H)-H(x,d_xu) -C_1\omega_{u,K'}(C_0s)\geq \frac{c(H)-H(x,d_xu)}2.$$
Combining this with ($**$) we obtain
$$\omega_{u,K'}^{-1}\left(\frac{c(H)-H(x,d_xu)}{2C_1}\right)\frac{c(H)-H(x,d_xu)}{2C_0}\leq
h_t(x,y)+c(H)t+u(x)-u(y).$$
This finishes the proof.
\end{proof}
\begin{proof}[Proof of Proposition \ref{propositionderivativeaction}]
We apply Lemma \ref{lemmemajorationdudefaut} above to obtain the compact set $K'$.
This lemma also gives for every $\ell\geq 0$ a constant $A=A(\ell)$ such that 
any {\rm C}$^{1,1}$ critical subsolution $u:M\to \R$ which is $\ell$-Lipschitz on $K'$ satisfies
$$\forall x\in K, \,\forall t\in [t_0,t'_0],\quad \frac{[c(H)-H(x,d_xu)]^2}A\leq h_t(x,x)+c(H)t.$$

To prove the second part, we will use the strict {\rm C}$^2$ convexity of $H$. Since the set $\{(x,p)\mid x\in K, H(x,p)\leq c(H)\}$ is compact the {\rm C}$^2$ strict convexity allows us 
to find $\beta>0$ such that for all $x\in K$, and $p_1,p_2\in T^*_xM$, with $H(x,p_i)\leq c(H)$, we have
$$H(x,p_2)-H(x, p_1)\geq \frac{\partial H(x,p_1)}{\partial p}(p_2-p_1)+\beta \lVert p_2-p_1\rVert_x^2.$$
Since $H$ is convex in $p$, for all $x\in K$, and $p_1,p_2\in T^*_xM$, with $H(x,p_i)\leq c(H)$, we also have $H(x,(p_1+p_2)/2)\leq c(H)$. therefore we can apply the above inequality to the pairs $((p_1+p_2)/2,p_1)$ and $((p_1+p_2)/2,p_2)$ to obtain
\begin{align*}
H(x,p_1)-H\left(x,\frac{p_1+p_2}2\right)&\geq \frac{\partial H(x,\frac{p_1+p_2}2)}{\partial p}(\frac{p_1-p_2}2)+\beta \lVert \frac{p_1-p_2}2\rVert_x^2\\
H(x,p_2)-H\left(x,\frac{p_1+p_2}2\right)&\geq \frac{\partial H(x,\frac{p_1+p_2}2)}{\partial p}(\frac{p_2-p_1}2)+\beta \lVert \frac{p_2-p_1}2\rVert_x^2.
\end{align*}
If we if we add these 2 inequalities, using $H(x,p_i)\leq c(H)$, and dividing by 2, we obtain
$$c(H)-H\left(x,\frac{p_1+p_2}2\right)\geq \beta \lVert \frac{p_2-p_1}2\rVert_x^2.$$
Therefore if $u_1,u_2:M\to \R$ are two {\rm C}$^1$ critical subsolutions  such that ($x\mapsto (x,d_xu_i), i=1,2$ have a Lipschitz constant  $\leq \ell$ on
$K'$, we get
\begin{equation*}
c(H)-H\left(x,\frac{d_xu_1+d_xu_2}2\right)\geq \beta \lVert \frac{d_xu_2-d_xu_1}2\rVert_x^2. \tag{$*$}
\end{equation*}
We denote by $T^*M\oplus T^*M$ the Whitney sum of $T^*M$ with itself (i.e. we consider the vector bundle over $M$ whose fiber at $x\in M$ is 
$T_x^*M\times T_x^*M$).  The maps
$$T^*M\to T^*M\oplus T^*M, \quad (x,p)\mapsto (x,p,0),$$
$$T^*M\to T^*M\oplus T^*M, \quad (x,p)\mapsto (x,0,p)$$
and
$$T^*M\oplus T^*M\to T^*M,
\quad (x,p_1,p_2)\mapsto \left(x,\frac{p_1+p_2}2\right)$$ are all {\rm C}$^\infty$. Therefore they are 
Lipschitz on any compact subset. Since for a
critical subsolution $u:M\to \R$ the values $(x,d_xu)$, for $ x\in K'$, 
are all in the compact subset $\{(x,p)\mid x\in K', H(x,p)\leq c(H)\}$, 
we can find a constant 
$B<\infty$ such that for  any two {\rm C}$^1$ critical subsolutions $u_1,u_2:M\to \R$ such that 
$x\mapsto (x,d_xu_i), i=1,2$ has a Lipschitz constant  $\leq \ell$ on $K'$, the map 
$x\mapsto(x, (d_xu_1+d_xu_2)/2)$ has Lipschitz constant $\leq B\ell$.
Since $(u_1+u_2)/2$ is also a critical subsolution, applying the first part of the proposition 
proved above with Lipschitz constant $\ell_1=B\ell$, we can find a constant $C_1$ such that
$$\forall x\in K,\,\forall t\in [t_0,t'_0], \quad c(H)-H\left(x,\frac{d_xu_1+d_xu_2}2\right)^2\leq C_1(h_t(x,x)+c(H)t).$$
Combining this inequality with ($*$) above we get
$$\forall x\in K,\,\forall t\in [t_0,t'_0],\quad \beta^2\lVert \frac{d_xu_2-d_xu_1}2\rVert_x^4\leq C_1^2(h_t(x,x)+c(H)t)^2.$$
This yields the second part of the Proposition with $C'=\beta^{-2}C_1^2$.
\end{proof}
We now can start the proof of Theorem \ref{THM2}. Let $\tilde{\mathcal A}^p$ be the set of points in the Aubry set $\tilde{\mathcal{A}}$ which are periodic but not fixed under the Euler-Lagrange flow $\phi^L_t$. This set projects on $\mathcal A^p$. Denote by $T:\tilde{\mathcal A}^p\to ]0,+\infty[$ the period map of Euler-Lagrange flow $\phi^L_t$, i.e. if $(x,v)\in \tilde{\mathcal A}^p$, the number $T(x)$ is the smallest positive number $t$ such that $\phi^L_t(x,v)=(x,v)$.
Using Proposition \ref{lemmepointsperiodiques} above, we can write $\tilde{\mathcal A}^p=\cup_{n\in \N}\tilde F_n$, with each $\tilde F_n$ compact and such that the restriction $T| \tilde F_n$ is continuous.
We denote by $F_n$ the projection of $\tilde F_n\subset TM$ on the base $M$.
We have ${\mathcal A}^p=\cup_{n\in \N} F_n$. If we want to show that 
$\Haus d({\mathcal A}^p,\delta_M)=0$, for some dimension $d> 0$, by the countable additivity of the Hausdorff measure in dimension $d$, it suffices to show that
$\Haus d(F_n,\delta_M)=0$, for every $n\in \N$.

Therefore from now on we fix some compact subset $\tilde F\subset \tilde{\mathcal A}^p$
on which the period map is continuous, and we will show that its Hausdorff measure in the appropriate dimension $d$ is $0$. We now perform one more reduction. In fact, we claim that it suffices for each $(x,v)\in \tilde F$ to find $\tilde S_{(x,v)}\subset TM$ a {\rm C}$^\infty$ 
 codimension 1 transversal section to the  Euler-Lagrange flow $\phi^L_t$, 
 such that $(x,v)\in \tilde  S_{(x,v)}$ and $\Haus d(\pi(\tilde F\cap\tilde  S_{(x,v)}),\delta_M)=0$, where $\pi:TM\to M$ is the canonical projection. Indeed, if this was the case, since, by transversality of $\tilde S$ to the flow $\phi^L_t$ the set $\tilde V_{x,v}=\cup_{t\in \R}\phi^L_t(\tilde  S_{(x,v)})$ is open in $TM$,  we could  cover the compact set $\tilde F$ by a finite number of sets $\tilde V_{(x_i,v_i)},i=1,\dots,\ell$. Note that by part (3) of Mather's theorem \ref{ThMather1}, the sets
 $\pi(\tilde F\cap \tilde V_{(x,v)})$ and $\pi(\tilde F\cap\tilde  S_{(x,v)})$ have the same image in the quotient Mather set, therefore we  get 
 $$\Haus d(\pi(\tilde F\cap \tilde V_{(x,v)}),\delta_M)=\Haus d(\pi(\tilde F\cap\tilde  S_{(x,v)}),\delta_M)=0.$$
Hence $F=\pi(\tilde F)$, which is covered by the finite number of sets $\pi(\tilde F\cap \tilde V_{(x_i,v_i)})$, does also satisfy $\Haus d(F,\delta_M)=0$.

Fix now $(x_0,v_0)$ in $\tilde F\subset \tilde\Aubry^p$. We proceed to construct the
transversal $\tilde S=\tilde  S_{(x_0,v_0)}$. We start with a {\rm C}$^\infty$ codimension 1 submanifold
$\tilde  S_0\subset TM$ which it transversal to the flow $\phi^L_t$, and which intersects the 
compact periodic  orbit $\phi^L_t(x_0,v_0)$ at exactly $(x_0,v_0)$. If $L$ (or $H$) is {\rm C}$^{k,1}$, the Poincar\'e first return time $\tau:\tilde  S_1\to ]0,\infty[$ on $T_0$ is defined and {\rm C}$^{k-1,1}$ on some smaller transversal $\tilde  S_1\subset \tilde  S_0$ containing $(x_0,v_0)$. We set $\theta: \tilde  S_1\to \tilde  S_0,
(x,v)\mapsto \phi^L_{\tau(x,v)}(x,v)$. This is the Poincar\'e return map, and it is also {\rm C}$^{k-1,1}$, as a composition of  {\rm C}$^{k-1,1}$ maps. Of course, we have $\tau(x_0,v_0)=T(x_0,v_0)$ and $\theta(x_0,v_0)=(x_0,v_0)$. Since $T$ is continuous on $F$, it is easy to show that $T=\tau$ and $\theta$ is the identity on $F\cap \tilde  S_2$, where $\tilde  S_2\subset \tilde  S_1$ is a smaller section containing $(x_0,v_0)$.

Pick $\epsilon >0$ small enough so that the radius of injectivity of the Riemannian manifold $M$ is $\geq \epsilon$  for every $x\in B_d(x_0,\epsilon)=\{y\in M\mid d(x_0,y)<\epsilon\}$, where $d$ is the distance obtained from the Riemannian metric on $M$. This implies that the
restriction of the square $d^2$ of the distance $d$ is of class {\rm C}$^\infty$ (like the Riemannian metric) on $ B_d(x_0,\epsilon/2)\times B_d(x_0,\epsilon/2)$.

We now take a smaller section $\tilde  S_3\subset \tilde  S_2$ around $(x_0,v_0)$ such that for every
$(x,v)\in \tilde  S_3$ both $x$ and $\pi\theta(x,v)$ of $M$ are in the ball $ B_d(x_0,\epsilon/2)$. This is possible by continuity since $\theta(x_0,v_0)=(x_0,v_0)$. For $(x,v)\in \tilde  S_3$, we set 
$$
\rho(x,v)= \tau(x,v)+d(\pi\theta(x,v),x).$$
We will now give an upper bound for
$h_{\rho(x,v)} (x,x) +c(H)\rho(x,v)$, when $(x,v)\in \tilde  S_3$.
For this we choose a loop $\gamma_{(x,v)}:[0,\rho(x,v)]\to M$ at $x$. This loop $\gamma_{(x,v)}$ is equal to  the curve 
$\gamma_{(x,v),1}(t)=\pi\phi_t(x,v)$ for $t\in [x,\tau(x,v)]$, which joins $x$ to 
$\pi\theta(x,v)$, followed by the shortest geodesic 
$\gamma_{(x,v),2}:[\tau(x,v),\rho(x,v)]\to M$, for the Riemannian metric, parametrized by arc-length and joining 
$\pi\theta(x,v)$ to $x$. Since  $\gamma_{(x,v),2}$ is parametrized by arc-length and is
 contained in $B_d(x_0,\epsilon)$, its action
is bounded by $Kd(\pi\theta(x,v),x)$, where $K=\sup \{L(x,v)\mid d(x,x_0)\leq \epsilon, \lVert v\rVert_x\leq 1\}<\infty$. On the other hand the action $a(x,v)$ of $\gamma_{(x,v),1}(t)$ is given by 
$$a(x,v)=\int_0^{\tau(x,v)}L[\phi^L_s(x,v)]\, ds.$$
Note that $a$ is also of class {\rm C}$^{k-1,1}$. It follows that, for $(x,v)\in \tilde  S_3$,
we have
$$h_{\rho(x,v)} (x,x) +c(H)\rho(x,v)\leq [a(x,v)+c(H)\tau(x,v)] +[K+c(H)]d(\pi\theta(x,v),x)).$$
Therefore if, for $(x,v)\in \tilde  S_3$,  we define 
$$
\Psi (x,v) = [a(x,v)+c(H)\tau(x,v)]^2 + d^2(\theta(x,v),x),
$$
we obtain 
$$\forall (x,v)\in \tilde  S_3, \quad 0\leq h_{\rho(x,v)} (x,x) +c(H)\rho(x,v)\leq [1+K+c(H)]\sqrt{\Psi (x,v)}.$$
Notice that $\Psi$ is {\rm C}$^{k-1,1}$ like $a$ and $\tau$, because $x,\pi\theta(x,v)\in B(x_0,\epsilon/2)$ and $d^2$ is {\rm C}$^\infty$ on the ball $B(x_0,\epsilon/2)$.
We now observe that $\Psi$ is identically $0$ on $\tilde F\cap \tilde  S_3$.
Indeed, for $(x,v)\in \tilde F\cap \tilde  S_3$, we have $\theta(x,v)=(x,v)$, therefore 
$d^2(\pi\theta(x,v),x)=0$. Moreover, since $(x,v)\in \tilde F\subset \tilde \Aubry$, the curve $t\mapsto\pi\phi^L_t(x,v)$ calibrates any critical subsolution $u:M\to \R$; in particular
\begin{align*}
u(\pi\phi^L_{\tau(x,v)}(x,v))-u(\pi(x,v))&=
\int_0^{\tau(x,v)} L\phi^L_s(x,v)\,ds+c(H)\tau(x,v)\\
&=a(x,v)+c(H)\tau(x,v).
\end{align*}
But $\phi^L_{\tau(x,v)}(x,v)=\theta(x,v)=(x,v)$ for $(x,v)\in \tilde F\cap \tilde  S_3$. Hence $a(x,v)+c(H)\tau(x,v)=0$, for $(x,v)\in \tilde F\cap \tilde  S_3$. Therefore $\Psi$ is identically $0$ on for $\tilde F\cap \tilde  S_3$.

To sum up the situation we have found two functions $\rho,\Psi:\tilde S_3\to[0,+\infty[$ such that:
\begin{enumerate}
\item the function $\rho$ is  continuous and $>0$ everywhere;
\item the function $\Psi$ is {\rm C}$^{k-1,1}$ and vanishes identically on 
$\tilde F\cap\tilde S_3$;
\item there exists a finite constant $C$ such that
$$\forall (x,v)\in \tilde S_3, \quad 0\leq  h_{\rho(x,v)} (x,x) +c(H)\rho(x,v)\leq C\sqrt{\Psi (x,v)}.$$
\end{enumerate}
This is all that we will use in the sequel of the proof.

We now fix a smaller Poincar\'e section $\tilde S_4$ containing $(x_0,v_0)$ whose closure $\operatorname{Cl}(\tilde S_4)$ is compact and contained in $\tilde S_3$. We now observe that  $K=\pi(\operatorname{Cl}(\tilde S_4))$ is a compact subset of $M$, and that 
$t_0=\min\{\tau(x,v)\mid (x,v)\in \operatorname{Cl}(\tilde S_4)\},t'_0=\max\{\tau(x,v)\mid (x,v)\in \operatorname{Cl}(\tilde S_4)\}$  are finite and $>0$ since $\tau$ is continuous and $>0$ on the compact set $\operatorname{Cl}(\tilde S_4)$. We can therefore apply Proposition 
\ref{propositionderivativeaction},  to obtain a set $K'$. We have to choose a constant $\ell$ needed to apply this Proposition \ref{propositionderivativeaction}. For this we invoke 
Theorem \ref{THMexistC11}: we can find a constant $\ell$ such that for any critical subsolution $u:M\to \R$ we can find a {\rm C}$^{1,1}$ critical subsolution $v:M\to \R$ which is equal to $u$ on the projected Aubry set $\Aub$ and such that $x\mapsto (x,d_xv)$ has Lipschitz constant $\leq \ell $ on $K'$. It follows from Lemma \ref{Fund} that
$$\forall x,y\in \Aub,\quad \delta_M(x,y)=\max \left\{ (u_1-u_2)(y) - (u_1-u_2)(x) \right\},$$
where the maximum is taken over all the pairs of {\rm C}$^{1,1}$ critical  subsolution $u_1,u_2:M\to\R$ such that $x\mapsto (x,d_xu_i),i=1,2$ have a Lipschitz constant $\leq \ell$ on $K'$. Using this $\ell$, we obtain, from  
Proposition \ref{propositionderivativeaction}, a constant $C'$ such that
$$\forall (x,v)\in \operatorname{Cl}(\tilde S_4),\quad \lVert d_xu_2-d_xu_1\rVert_x^4
\leq C' [h_{\tau(x,v)}(x,x)+c(H)\tau(x,v)], $$
for  every pair of {\rm C}$^{1,1}$ critical  subsolution $u_1,u_2:M\to\R$ such that $x\mapsto (x,d_xu_i),i=1,2$ have a Lipschitz constant $\leq \ell$ on $K'$. Therefore by the properties of $\tau$ and $\Psi$ explicited above we obtain
$$\forall (x,v)\in \operatorname{Cl}(\tilde S_4),\quad \lVert d_xu_2-d_xu_1\rVert_x
\leq C \Psi(x,v)^{1/8}, $$
again for  every pair of {\rm C}$^{1,1}$ critical  subsolution $u_1,u_2:M\to\R$ such that $x\mapsto (x,d_xu_i),i=1,2$ have a Lipschitz constant $\leq \ell$ on $K'$.
Since $\Psi$ is of class {\rm C}$^{k-1,1}$ and is identically $0$ on $\tilde F\cap \tilde S_4$, we can invoke
Lemma \ref{lemmaMorseVanish} to obtain a decomposition
$$
\tilde F\cap \tilde S_4 =\cup_{i \in \N} \tilde{A}_i,
$$
 with $A_i$ a compact subset, a family $(B_i)_{i\in \N}$of {\rm C}$^1$ compact embedded discs in  $\tilde S_4$
 and  constants $(C_i)_{i\in \N}$ such that
\begin{equation*}
\forall (x,v)\in A_i,\, \forall (y,w)\in B_i,\quad \Psi(y,w)=\lvert \Psi(y,w)-\Psi(x,w)\rvert\leq 
C_i \tilde d[(y,w),(x,v)]^k,
\end{equation*}
where $\tilde d$ is the distance obtained from a fixed Riemannian metric on 
$\tilde S_4$.
Combining with what we obtained above, we find constants $C'_i$ (independent of the pair of functions $u_1,u_2$) such that
\begin{equation*}
\forall (x,v)\in A_i,\,\forall (y,w)\in B_i,\quad \lVert d_yu_1-d_yu_2\rVert_x\leq 
C'_i \tilde d[(y,w),(x,v)]^{k/8}.
\end{equation*}
Since we want to consider $u_1$ and $u_2$ as function on $B_i\subset TM$
composing with $\pi:TM\to M$, we can rewrite this as
\begin{equation*}
\forall (x,v)\in A_i,\,\forall (y,w)\in B_i, \quad
\lVert d_{(y,w)}u_1\circ \pi-d_{(y,w)}u_2\circ\pi\rVert_x\leq 
C'_i \tilde d[(y,w),(x,v)]^{k/8}.
\end{equation*}

Again as in the proof of the previous theorem to simplify things we can identify
$B_I$ with a Euclidean ball $\BB_i$ of some dimension, and since the identification  is done by a {\rm C}$^1$ diffeomorphism, we can find constants $C_i''$ (independent of the pair of functions $u_1,u_2$)
such that 
\begin{equation*}
\forall (x,v)\in A_i,\,\forall (y,w)\in \BB_i, \quad
\lVert d_{(y,w)}u_1\circ \pi-d_{(y,w)}u_2\circ\pi\rVert_{\rm euc}\leq 
C_i''\Vert(y,w)-(x,v)\rVert_{\rm euc}^{k/8}.
\end{equation*}
In the Euclidean disc $\BB_i$, we can integrate this inequality along the Euclidean segment joining $(x,v)$ to $(y,w)$ to obtain
\begin{multline*}
\forall (x,v)\in A_i,\,\forall (y,w)\in \BB_i, \quad
\lvert (u_1-u_2)(y)-(u_1-u_2)(x)\rvert\\
\leq \frac{C_i''}{1+(k/8)}\Vert(y,w)-(x,v)\rVert_{\rm euc}^{1+(k/8)}.
\end{multline*}
Of course, since the identification of $B_i\subset TM$ with $\BB_i$ is done by a {\rm C}$^1$ diffeomorphism changing constants again to some $\tilde C_i$ (independent of the pair of functions $u_1,u_2$), we get
\begin{equation*}
\forall (x,v)\in A_i,\,\forall (y,w)\in \BB_i, \quad
\lvert (u_1-u_2)(y)-(u_1-u_2)(x)\rvert\leq 
\tilde C_id[(y,w),(x,v)]^{1+(k/8)},
\end{equation*}
where  $d$ is a distance on $TM$ obtained from a Riemannian metric.
Observe now that, by Mather's theorem, the projection $\pi :\tilde\Aub\to\Aub$ is  bijective with an inverse which is locally Lipschitz. Therefore since $A_i$ is compact  and contained in $\tilde F\subset \tilde \Aub$, changing again the constants to $\tilde C'_i$ (independent of the pair of functions $u_1,u_2$), we obtain
\begin{equation*}
\forall x,y\in \pi(A_i),\quad
\lvert (u_1-u_2)(y)-(u_1-u_2)(x)\rvert\leq 
\tilde C'_id[y,x]^{1+(k/8)},
\end{equation*}
Since this inequality is true now for every pair f of {\rm C}$^{1,1}$ critical  subsolution $u_1,u_2:M\to\R$ such that $x\mapsto (x,d_xu_i),i=1,2$ have a Lipschitz constant $\leq \ell$ on $K'$ (with the constant $\tilde C_i'$ independent of the pair of functions $u_1,u_2$), we conclude that
$$\forall x,y\in\pi(A_i),\quad \delta_M(x,y)\leq 
\tilde C''_id[y,x]^{1+(k/8)} .$$
Therefore by Lemma \ref{lemNEW2} we obtain that
$$\Haus {8\dim M/(k+8)}(\pi(A_i))=0.$$
Again by countable additivity this gives  $\Haus {8\dim M/(k+8)}(\pi(\tilde F \cap\tilde S_4))=0$. This finishes the proof of the Theorem.

\begin{remark}
We observe that, from our proof, for any $\tilde F \subset \tilde{\mathcal A}$, the semi-metric space $(\pi(\tilde F),\d_{M})$ has vanishing one-dimensional Hausdorff measure as soon as the following properties are satisfied: there are $r>0$, $k',l \in \N$ and a function $G:TM \rightarrow \R$ of class {\rm C}$^{k',1}$ such that
\begin{enumerate}
\item $G(x,v) \equiv 0$ on $\tilde F$,
\item $\left\{ m_{r}(x) \right\}^l  \leq G(x,v)$ for all  $(x,v) \in TM$, 
\item $k' \geq 4l(\dim M-1)-1$,
\end{enumerate}
where $m_{r} (x) = \inf_{t \geq r} \left\{h_t(x,x)+c(H)t \right\}$.
\end{remark}

\begin{remark}
By Proposition \ref{HolderOrdre2}, for every compact subset $K \subset M$ there is a constant $C_K>0$ such that 
$$
\forall x\in K, \quad 
h(x,x) \leq C_K d(x,\mathcal{A})^2,
$$
where $d(x,\mathcal{A})$ denotes the Riemannian distance from $x$ to the set $\Aub$ (which is assumed to be nonempty).
Therefore, from the remark above, we deduce that if there are $l\in \N$ and a function $G:M \rightarrow \R$ of class {\rm C}$^{k',1}$ with $k' \geq 2l(\dim M -1) -1$ such that
$$
\forall x \in M, \quad
d(x,\mathcal{A})^l \leq G(x),
$$
then $(\mathcal{A}_M,d_{M})$ has vanishing one-dimensional Hausdorff measure.
\end{remark}

\subsection{Proof of Theorem \ref{THMSURF}}
By Theorems \ref{THM2} and \ref{THM3} we know that $(\mathcal{A}_M^0 \cup \Aub_M^p,\delta_M)$ has
zero Hausdorff dimension. Thus the result will follow once we will show that
$\Aub_M \setminus (\mathcal{A}_M^0\cup \Aub_M^p)$ is a finite set.

We recall that the Aubry set $\tilde{\mathcal{A}} \subset TM$ is given by the set of $(x,v) \in TM$
such that $x\in \mathcal{A}$ and $v$ is the unique $v \in T_xM$ such that $d_xu =\frac{\partial L}{\partial v}(x,v)$
for any critical viscosity subsolution.
This set is invariant under the Euler-Lagrange flow $\phi_t^L$. For every $x\in \mathcal{A}$,
we denote by $\mathcal O(x)$ the projection on $\Aub$ of the orbit of $\phi_t^L$ which passes through $x$.
We observe that by Theorem \ref{ThMather1} (3) the following simple fact holds:
\begin{lemma}
\label{simplelemma}
If $x,y \in \Aub$ and $\ov{\mathcal O(x)} \cap \ov{\mathcal O(y)} \neq \emptyset$,
then $\delta_M(x,y)=0$.
\end{lemma}

Let us define
$$
\mathcal C_0=\{x \in \Aub \mid \ov{\mathcal O(x)} \cap \Aub_0\},
\quad
\mathcal C_p=\{x \in \Aub \mid \ov{\mathcal O(x)} \cap \Aub_p\}.
$$
Thus, if $x \in \mathcal C_0 \cup \mathcal C_p$,
by Lemma \ref{simplelemma} the Mather distance between $x$ and $\mathcal{A}^0\cup \Aub^p$ is $0$, and we have done.

Let us now define $\mathcal C=\Aub \setminus(\mathcal C_0 \cup \mathcal C_p)$, and let $(\mathcal C_M, \delta_M)$
be the quotiented metric space. To conclude the proof, we show that this set consists of a finite number of points.

Let $u$ be a {\rm C}$^{1,1}$ critical subsolution (whose existence is provided by \cite{bernard}), and let $X$
be the Lipschitz vector field uniquely defined by the relation
$$
\mathcal{L}(x,X(x))=(x,d_xu),
$$
where $\mathcal L$ denotes the Legendre transform.
Its flow extends on the whole manifold the flow considered above on $\Aub$.
We fix $x \in \mathcal C$. Then $\ov{\mathcal O(x)}$ is a non-empty, compact, invariant set which contains
a non-trivial minimal set for the flow of $X$ (see \cite[Chapter 1]{NZ}).
By \cite{mark70}, we know that there exists at most a finite number of such non-trivial minimal sets.
Therefore, again by Lemma \ref{simplelemma}, $(\mathcal C_M, \delta_M)$ consists only in the finite number of points.

\section{Applications in Dynamics}
Throughout this section, $M$ is assumed to be compact. As before, $H:T^*M \rightarrow \R$ is an Hamiltonian of class at least {\rm C}$^2$ satisfying the three usual conditions (H1)-(H3), and $L$ is the Tonelli Lagrangian which is associated to it by Fenchel's duality. 

As in Subsection \ref{MatherDistAndSubsol}, we denote $\mathcal{SS}$ the set of critical viscosity subsolutions and by ${\cal S}_-$ the set of critical viscosity (or weak KAM) solutions, so that ${\cal S}_- \subset \mathcal{SS}$.

\subsection{Some more facts about the Aubry set when the manifold is compact}
From the characterization of the Aubry set given by Theorem \ref{DescriptionAubry}, it is natural to introduce
the Ma\~n\'e set $\tilde {\cal N}$ given by 
$$
\tilde {\cal N}=\bigcup_{u\in {\cal {SS}}}\tilde{\cal I}(u).
$$
Like for  $\tilde {\cal A}$, the subset  $\tilde {\cal N}$ of $TM$ is compact and invariant under the Euler-Lagrange flow $\phi^L_t$ of $L$.

\begin{theorem}[Ma\~n\'e]
\label{mane}
When $M$ is compact, each point of the invariant set $\tilde {\cal A}$ is chain-recurrent for the restriction $\phi_t^L|_{\tilde {\cal A}}$. Moreover, the invariant set $\tilde {\cal N}$ is chain-transitive for the restriction $\phi_t^L|_{\tilde {\cal N}}$.
\end{theorem}

\begin{corollary}
\label{corDyn1}
When $M$ is compact, the restriction $\phi_t^L|_{\tilde {\cal A}}$ to the invariant subset $ \tilde {\cal A}$ is chain-transitive if and only if $ \tilde {\cal A}$ is connected.
\end{corollary}

\begin{proof} This is an easy well-known result in the theory of Dynamical Systems: Suppose $\theta_t$, $t\in\R$, is a flow on the compact metric space $X$. If every point of $X$ is chain-recurrent for $\theta_t$, then $\theta_t$ is chain-transitive if and only if $X$ is connected.
\end{proof}
For the following result see \cite{fathibook}  or \cite[Th\'eor\`eme 1]{fathi97b}.

\begin{theorem}
\label{uniqvisco} When $M$ is compact, the following properties are satisfied:\\
1) Two weak KAM solutions that coincide on $\cal A$ are equal everywhere.\\
2) For every $u\in {\cal {SS}}$, there is a unique weak KAM solution  $u_-:M\to\R$ 
such that $u_-=u$ on ${\cal A}$; moreover, the two function $u$ and $u_-$ are also equal on ${\cal I}(u)$.
\end{theorem}

It follows from the second statement in this theorem that we have
$$
\tilde {\cal N}=\bigcup_{u\in {\cal {S}_-}}\tilde{\cal I}(u).
$$
Moreover, it can be easily shown from the results of \cite{fathibook} that
$$
{\cal A}=\bigcap_{u\in {\cal {S}_-}}{\cal I}(u).
$$

We give now the general relationship between uniqueness of weak KAM solutions and the quotient Mather set.

\begin{proposition}
\label{propDyn2}
Suppose $M$ is compact. The following two statements are equivalent:
\begin{itemize}
\item[1)] Any two weak KAM solutions differ by a constant.
\item[2)] The Mather quotient $(A_M,\delta_M)$ is trivial, i.e. is reduced to one point.
\end{itemize}
Moreover, if anyone of these conditions is true, then $\tilde{\cal A}=\tilde{\cal N}$, and therefore $\tilde{\cal A}$ is connected and the restriction of the Euler-Lagrange flow $\phi_t^L$ to $\tilde{\cal A}$ is chain-transitive.
\end{proposition}

\begin{proof} 
For every fixed $x\in M$, the function $y\mapsto h(x,y)$ is a weak KAM solution. Therefore if we assume that any two weak KAM solutions differ by a constant, then for $x_1,x_2\in M$ we can find a constant $C_{x_1,x_2}$ such that
$$
\forall y\in M, \quad h(x_1,y)=C_{x_1,x_2}+h(x_2,y).
$$
If $ x_2\in {\cal A}$, then $h(x_2,x_2)=0$, therefore evaluating the equality above for $y=x_2$, we obtain $C_{x_1,x_2}=h(x_1,x_2)$. Substituting in the equality and evaluating we conclude
$$
\forall x_1\in M,\, \forall x_2\in{\cal A},\quad h(x_1,x_1)=h(x_1,x_2)+h(x_2,x_1).
$$
This implies
$$
\forall x_1, x_2\in{\cal A}, \quad  h(x_1,x_2)+h(x_2,x_1)=0,
$$
which means that $\delta_M(x_1,x_2)=0$ for every $x_1,x_2 \in \mathcal{A}$.

To prove the converse, let us recall that for every critical subsolution $u$, we have
$$
\forall x,y\in M, \quad u(y)-u(x)\leq h(x,y).
$$
Therefore applying this for a pair $u_1,u_2\in{\cal SS}$, we obtain
\begin{align*}
\forall x,y \in M, \quad & u_1(y)-u_1(x)\leq h(x,y),\\
&u_2(x)-u_2(y)\leq h(y,x).
\end{align*}
Adding and rearranging, we obtain
$$
\forall x,y \in M, \quad (u_1-u_2)(y)-(u_1-u_2)(x)\leq h(x,y)+ h(y,x).
$$
Since the right hand side is symmetric in $x,y$, we obtain
$$
\forall x,y \in M, \quad \rvert(u_1-u_2)(y)-(u_1-u_2)(x)\lvert\leq h(x,y)+ h(y,x).
$$
If we assume that 2) is true, this implies that $u_1-u_2$ is constant $c$ on the projected Aubry set ${\cal A}$,
that is $u_1=u_2+c$ on ${\cal A}$. Thus, if $u_1, u_2$ are weak KAM solutions, then we have $u_1=u_2+c$ on $M$, because any two solutions equal on the Aubry set are equal everywhere by 2) of Theorem \ref{uniqvisco}.

It remains to show the last statement. Notice that if $u_1,u_2\in {\cal {SS}}$ differ by a constant then  $\tilde{\cal I}(u_1)=\tilde{\cal I}(u_2)$. Therefore if any two elements in ${\cal S}_-$ differ by a constant, then 
$$
\tilde{\cal A}=\tilde{\cal I}(u)=\tilde{\cal N},
$$
where $u$ is any element in ${\cal S}_-$. But, by Ma\~n\'e's Theorem \ref{mane}, the invariant set $\tilde{\cal N}$ is chain-transitive for the flow $\phi_t$, hence it is connected by Corollary \ref{corDyn1}.
\end{proof}

We now denote by $X_L$ the Euler-Lagrange vector field of $L$, that is the vector field on $TM$ that generates $\phi_t^L$.
We recall that an important property of $X_L$ is that
$$\forall (x,v)\in TM,\quad T\pi(X_L(x,v))=v,$$
where $T\pi: T(TM) \rightarrow TM$ denotes the canonical projection.\\
Here is a last ingredient that we will have to use.

\begin{proposition}[Lyapunov Property]
\label{lyapunov}
Suppose $ u_1,u_2\in {\cal {SS}}$. The function $(u_1-u_2)\circ \pi$ is non-decreasing along any orbit of the Euler Lagrange flow $\phi_t^L$ contained in $\tilde{\cal I}(u_2)$. If we assume $u_1$ is differentiable at $x\in {\cal I}(u_2)$, and $(x,v)\in \tilde{\cal I}(u_2)$, then, using that $u_2$ is differentiable on ${\cal I}(u_2)$, we obtain
$$X_L\cdot[(u_1-u_2)\circ \pi](x,v)=d_xu_1(v)-d_xu_2(v)\leq 0.$$
Moreover, the inequality above is an equality, if and only if $d_xu_1=d_xu_2$. In that case $H(x,d_xu_1)=H(x,d_xu_2)=c(H)$.
\end{proposition}

\begin{proof} 
If $(x,v)\in \tilde{\cal I}(u_2)$ then $t\mapsto \pi\phi_t(x,v)$ is 
$(u_2,L,c(H))$-calibrated, hence
$$
\forall t_1\leq t_2, \quad u_2\circ\pi(\phi_{t_2}(x,v))-u_2\circ\pi(\phi_{t_1}(x,v))= \int_{t_1}^{t_2}L(\phi_s(x,v))\,ds+c(H)(t_2-t_1).
$$
Since $u_1\in {\cal {SS}}$, we get
$$
\forall t_1\leq t_2, \quad u_1\circ\pi(\phi_{t_2}(x,v))-u_1\circ\pi(\phi_{t_1}(x,v))\leq\int_{t_1}^{t_2}L(\phi_s(x,v))\,ds+c(H)(t_2-t_1).
$$
Combining these two facts, we conclude
$$
\forall t_1\leq t_2, \quad u_1\circ\pi(\phi_{t_2}(x,v))-u_1\circ\pi(\phi_{t_1}(x,v))\leq u_2\circ\pi(\phi_{t_2}(x,v))-u_2\circ\pi(\phi_{t_1}(x,v)).
$$
This implies
$$
\forall t_1\leq t_2, \quad (u_1-u_2)\circ \pi(\phi_{t_2}(x,v))\leq(u_1-u_2)\circ \pi(\phi_{t_1}(x,v)).
$$
Recall that $u_2$ is differentiable at every $x\in {\cal I}(u_2)$. Thus, if also $d_xu_1$ exists, if $(x,v)\in \tilde{\cal I}(u_2)$ we obtain
$$
X_L\cdot[(u_1-u_2)\circ \pi](x,v)\leq 0.
$$
We remark that $X_L\cdot[(u_1-u_2)\circ \pi](x,v)=d_x(u_1-u_2)(T\pi\circ X_L(x,v))$. Since $T\pi\circ X_L(x,v)=v$, we obtain
$$
X_L\cdot[(u_1-u_2)\circ \pi](x,v)=d_xu_1(v)-d_xu_2(v)\leq 0.
$$
If the last inequality is an equality, we get $d_xu_1(v)=d_xu_2(v)$. Since 
$(x,v)\in \tilde{\cal I}(u_2)$, we have $d_xu_2=\frac{\partial L}{\partial v}(x,v)$ and $H(x,d_xu_2)=c(H)$, therefore the Fenchel inequality yields the equality
$$
d_xu_2(v)=L(x,v)+H(x,d_xu_2)=L(x,v)+c(H).
$$
Since $u_1\in{\cal {SS}}$, we know that $H(x,d_xu_1)\leq c(H)$. The previous equality, using the Fenchel inequality $d_xu_1(v)\leq L(x,v)+H(x,d_xu_1)$, and the fact that $d_xu_1(v)=d_xu_2(v)$, implies
$$
H(x,d_xu_1)=c(H) \quad \text{ and } \quad d_xu_1(v)=L(x,v)+H(x,d_xu_1).
$$
This means that we have equality in the Fenchel inequality $d_xu_1(v)\leq L(x,v)+H(x,d_xu_1)$, we therefore conclude that $d_xu_1=\frac{\partial L}{\partial v}(x,v)$, but the right hand side of this last equality is $d_xu_2$.
\end{proof} 

\subsection{Mather disconnectedness condition}

\begin{definition}{\rm We will say that the the Tonelli Lagrangian $L$ on $M$ satisfies the \textit{Mather disconnectedness condition} if for every pair $u_1,u_2\in {\cal S}_-$, the image $(u_1-u_2)({\cal A})\subset\R$ is totally disconnected.}
\end{definition}

Notice that by part 2) of Theorem \ref{uniqvisco}, if $L$ satisfies the Mather disconnectedness condition, then for every pair of critical sub-solutions $u_1,u_2$, the image $(u_1-u_2)({\cal A})\subset\R$ is also totally disconnected.
\begin{proposition}
\label{propHausdisc}
If $\Haus1(\mathcal{A}_M,\delta_M)=0$, then $L$ satisfies the Mather disconnectedness condition.
\end{proposition}
\begin{proof}
If $u_1, u_2 \in {\cal SS}$, $u_1-u_2$ is $1$-Lipschitz with respect to $\d_M$, see the proof of Proposition \ref{propDyn2}.
Therefore the $1$-dimensional Hausdorff measure (i.e. Lebesgue measure) of $(u_1-u_2)({\cal A})$ is $0$ like $\Haus1(\mathcal{A}_M,\delta_M)$.
The result follows since a subset of $\R$ of Lebesgue measure $0$ is totally disconnected.
\end{proof}

By Proposition \ref{propHausdisc}, the results obtained above in this work contain the following theorem.

\begin{theorem}
\label{TheoremePrincipal}
Let $L$ be a Tonelli Lagrangian on the compact manifold $M$, it satisfies the Mather disconnectedness condition in the following five cases:
\begin{itemize}
\item[(1)] The dimension of $M$ is $1$ or $2$.
\item[(2)] The dimension of $M$ is $3$, and $\tilde{\cal A}$ contains no fixed point of the Euler-Lagrange flow.
\item[(3)] The dimension of $M$ is $3$, and $L$ is of class {\rm C}$^{3,1}$.
\item[(4)] The Lagrangian is of class {\rm C}$^{k,1}$, with $k\geq2\dim M-3$, and every point of $\tilde{\cal A}$ is fixed under the Euler-Lagrange flow $\phi_t^L$.
\item[(5)] The Lagrangian is of class {\rm C}$^{k}$, with $k\geq8\dim M-7$, and each point of $\tilde{\cal A}$ either is fixed under the Euler-Lagrange flow $\phi_t^L$ or its orbit in the Aubry set is periodic with (strictly) positive period.
\end{itemize}
\end{theorem}

\begin{lemma}
\label{lemmefond}
Suppose that $L$ is a Tonelli Lagrangian $L$ on the compact manifold $M$ that satisfies the Mather disconnectedness condition. For every $u\in {\cal{SS}}$, the set of points in $\tilde{\cal I}(u)$ which are chain-recurrent for the restriction
$\phi_t^L|_{\tilde{\cal I}(u)}$ of the Euler-Lagrange flow is precisely the Aubry set $\tilde{\cal A}$.
\end{lemma}

\begin{proof} 
First of all, we recall that, from Theorem \ref{mane}, each point of $\mathcal{A}$ is chain-recurrent for the restriction $\phi_t^L|_{\tilde {\cal A}}$. By \cite[Theorem 1.5]{fs04}, we can find a {\rm C}$^1$ critical viscosity subsolution $u_1:M\to\R$ which is strict outside $\cal A$, i.e. for every $x\notin {\cal A}$ we have $H(x,d_xu_1)<c(H)$.
We define $\theta$ on $TM$ by $\theta=(u_1-u)\circ\pi$. By Proposition  \ref{propertiesI}, we know that at each point $(x,v)$ of $\tilde{\cal I}(u)$ the derivative of $\theta$ exists and depends continuously on $(x,v)\in\tilde{\cal I}(u)$. By Proposition \ref{lyapunov},
at each point of $(x,v)$ of $\tilde{\cal I}(u)$, we have
$$X_L\cdot\theta(x,v)=d_xu_1(v)-d_xu(v))\leq 0,$$
with the last inequality an equality if and only if $d_xu_1=d_xu$, and this implies
$H(x,d_xu_1)=c(H)$. Since $u_1$ is strict outside $\cal A$, we conclude that $X_L\cdot\theta<0$ on $\tilde{\cal I}(u)\setminus \tilde{\cal A}$.
Suppose that $(x_0,v_0)\in\tilde{\cal I}(u)\setminus \tilde{\cal A}$. By invariance of both $\tilde{\cal A}$ and $\tilde{\cal I}(u)$, every point on the orbit $\phi_t^L(x_0,v_0), t\in \R$ is also contained in  $\tilde{\cal I}(u)\setminus \tilde{\cal A}$, therefore $t\mapsto c(t)=\theta(\phi_t(x_0,v_0))$ is (strictly) decreasing , and so we have $c(1)<c(0)$. Observe now that
$\theta(\tilde{\cal A})=(u_1-u)({\cal A})$ is totally disconnected by the Mather disconnectedness condition.
Therefore we can find $c\in ]c(1),c(0)[\setminus \theta(\tilde{\cal A})$. By what we have seen, the directional derivative $X_L\cdot \theta$ is $<0$ at every point of the level set $L_c=\{(x,v)\in \tilde{\cal I}(u)\mid \theta(x,v)=c\}$.  Since $\theta$ is everywhere non-increasing on the orbits of $\phi_t^L$ and $X_L\cdot \theta<0$ on $L_c$, we get
$$
\forall t>0, \quad \forall (x,v)\in L_c, \quad \theta(\phi_t(x,v))<c.
$$
Consider the compact set $K_c=\{(x,v)\in \tilde{\cal I}(u)\mid \theta(x,v)\leq c\}$. Using again that $\theta$ is non-increasing on the orbits of $\phi_t^L|_{\tilde{\cal I}(u)}$, we have
$$
\forall t\geq 0, \quad \phi_t^L(K_c)\subset K_c \quad \text{ and } \quad \phi_t^L(K_c\setminus L_c)\subset K_c\setminus L_c.
$$
Using what we obtained above on $L_c$, we conclude that 
$$
\forall t>0, \quad \phi_t^L(K_c)\subset K_c\setminus L_c.
$$
We now fix some metric on $\tilde{\cal I}(u)$ defining its topology. We then consider the compact set $\phi_1^L(K_c)$. It is contained in the open set $K_c\setminus L_c=
\{(x,v)\in \tilde{\cal I}(u)\mid \theta(x,v)<c\}$. We can therefore find $\epsilon>0$ such that the $\epsilon$ neighborhood $V_\epsilon(\phi_1(K_c))$ of $\phi_1^L(K_c)$ in $\tilde{\cal I}(u)$ is also contained in $K_c$. Since for $t\geq 1$ we have $\phi_{t-1}^L(K_c)\subset K_c$, and therefore $\phi_{t}^L(K_c)\subset\phi_1(K_c)$, it follows that
$$
V_\epsilon\left(\bigcup_{t\geq 1}\phi_t^L(K_c)\right)\subset K_c.
$$
It is know easy to conclude that every $\epsilon$-pseudo orbit for $\phi_t^L|_{\tilde{\cal I}(u)}$ that starts in $K_c$ remains in $K_c$. Since $\theta(\phi_1^L(x_0,v_0))=c(1)<c<c(0)=
\theta(x_0,v_0)$, no $\alpha$-pseudo orbit starting at $(x_0,v_0)$ can return to $(x_0,v_0)$, for $\alpha\leq \epsilon$ such that the ball of center $\phi_1^L(x_0,v_0)$ and radius $\alpha$, in $\tilde{\cal I}(u)$, is contained in $K_c$. Therefore $(x_0,v_0)$ cannot be chain recurrent.
\end{proof}

\begin{theorem}
\label{appliDyn}
Let $L$ be a Tonelli Lagrangian on the compact manifold $M$. If $L$ satisfies the Mather disconnectedness condition, then the following statements are equivalent:
\begin{itemize}
\item[(1)] The Aubry set $\tilde{\cal A}$, or its projection $\cal A$, is connected.
\item[(2)] The Aubry set $\tilde{\cal A}$ is chain-transitive for the restriction of the Euler-Lagrange flow $\phi_t^L|_{\tilde{\cal A}}$.
\item[(3)] Any two weak KAM solutions differ by a constant.
\item[(4)] The Aubry set $\tilde{\cal A}$ is equal to the Ma\~n\'e set $\tilde{\cal N}$.
\item[(5)] There exists $u\in {\cal{SS}}$ such that $\tilde{\cal I}(u)$ is chain-recurrent for the restriction $\phi_t|_{\tilde{\cal I}(u)}$ of the Euler-Lagrange flow.
\end{itemize}
\end{theorem}

\begin{proof}
From Corollary \ref{corDyn1}, we know that (1) and (2) are equivalent. 
\par
If (1) is true then for $u_1,u_2\in {\cal S}_-$, the image $u_1-u_2({\cal A})$ is a sub-interval of $\R$, but by the Mather disconnectedness condition, it is also totally disconnected,
therefore $u_1-u_2$ is constant. Hence (1) implies (3).
\par
If (3) is true then (4) follows from Proposition \ref{propDyn2}.
\par
Suppose now that (4) is true. Since for every $u\in{\cal {SS}}$, we have $\tilde{\cal A}\subset \tilde{\cal I}(u)\subset \tilde {\cal N}$, we obtain $\tilde{\cal I}(u)=\tilde {\cal N}$. 
But $\tilde {\cal N}$ is chain-transitive for the restriction $\phi_t^L|_{\tilde {\cal N}}$. Hence
(4) implies (5).
\par
If (5) is true for some $u\in {\cal{SS}}$, then every point of  $\tilde{\cal I}(u)$ 
is chain-recurrent for the restriction $\phi_t^L|_{\tilde{\cal I}(u)}$.
Lemma \ref{lemmefond} then implies that $\tilde{\cal A}=\tilde{\cal I}(u)$, and we therefore satisfy (2).
\end{proof}

\begin{remark}{\rm For each integer $d>0$, and each $\epsilon>0$, John Mather has constructed on the torus $\torus^d=\R^d/\Z^d$ a Tonelli Lagrangian $L$ of class {\rm C}$^{2d-3,1-\epsilon}$ such that $\tilde{\cal A}$ is connected, contained in the fixed points of the Euler-Lagrange flow, and the Mather quotient $({\cal A}_M,\delta_M)$ is isometric to an interval, see \cite{mather04}. In particular for such a Lagrangian, Theorem \ref{appliDyn} cannot be true.}
\end{remark}

\subsection{Ma\~n\'e Lagrangians}

We give know an application to the Ma\~n\'e example associated to a vector field. Suppose $M$ is a compact Riemannian manifold, where the metric $g$ is of class {\rm C}$^\infty$. If $X$ is a {\rm C}$^k$ vector field on $M$, with $k\geq 2$, we define the Lagrangian $L_X:TM\to \R$ by 
$$L_X(x,v)=\frac12\lVert v-X(x)\rVert_x^2,$$
where as usual $\lVert v-X(x)\rVert_x^2=g_x(v,v)$.
We will call $L_X$ the Ma\~n\'e Lagrangian of $X$, see the Appendix in \cite{mane92}. The following proposition gives the obvious properties of $L_X$.

\begin{proposition}
\label{propLagMane}
Let $L_X$ the Ma\~n\'e Lagrangian of the {\rm C}$^k$ vector field $X$, with $k\geq 2$, on the compact Riemannian manifold $M$. We have
$$
\frac{\partial L_X}{\partial v}(x,v)=g_x( v-X(x),\cdot).
$$
 Its associated Hamiltonian $H_X:T^*M\to\R$ is given by 
$$H_X(x,p)=\frac12 \lVert p\rVert_x^2+p(X(x)).$$
The constant functions are solutions of the Hamilton-Jacobi equation
$$
H_X(x,d_xu)=0.
$$
Therefore, we obtain $c(H)=0$. Moreover, we have
$$
\tilde{\cal I}(0)=\operatorname{Graph}(X)=\{(x,X(x))\mid x\in M\}.
$$
If we call $\phi_t$ the Euler-Lagrange flow of $L_X$ on $TM$, then for every $x\in M$, and every $t\in\R$, we have $\phi_t(x,X(x))=(\gamma_x^X(t), \dot\gamma_x^X(t))$, where $\gamma_x^X$ is the solution of the vector field $X$ which is equal to $x$ for $t=0$. In particular, the restriction $\phi_t|_{\tilde{\cal I}(0)}$ of the Euler-Lagrange flow to $\tilde{\cal I}(0)=\operatorname{Graph}(X)$ is conjugated (by $\pi|_{\tilde{\cal I}(0)}$) to the flow of $X$ on $M$.
\end{proposition}

\begin{proof} The computation of $\partial L_X/\partial v$ is easy. For $H_X$, we recall that  $H_X(x,p)=p(v_p)-L(x,v_p)$, where $v_p\in T_xM$ is defined by 
$p=\partial L_X/\partial v(x,v_p)$. Solving for $v_p$, and substituting yields the result.

If $u$ is a constant function then $d_xu=0$ everywhere, and obviously $H_X(x,d_xu)=0$. The fact that $c(H)=0$ follows, since $c(H)$ is the only value $c$ for which there exists
a viscosity solution of the Hamilton-Jacobi equation $H(x,d_xu)=c$.

Let us define $u_0$ as the null function on $M$. Suppose now that $\gamma:(-\infty,+\infty)\to M$ is a solution of $X$ (by compactness of $M$, solutions of $X$ are defined for all time). We have $d_{\gamma(t)}u_0(\dot\gamma(t))=0$, and $H_X(\gamma(t),d_{\gamma(t)}u_0)=0$; moreover, since $\dot\gamma(t)=X(\gamma(t))$, we also get 
$L_X(\gamma(t),\dot\gamma(t))=0$. It follows that
$$d_{\gamma(t)}u_0(\dot\gamma(t))=L_X(\gamma(t),\dot\gamma(t))+H_X(\gamma(t),d_{\gamma(t)}u_0)=L_X(\gamma(t),\dot\gamma(t)).$$
By integration, we see that $\gamma$ is $(u_0,L_X,0)$-calibrated, therefore it is an extremal. Hence we get $\phi_t(\gamma(0),\dot\gamma(0))=(\gamma(t),\dot\gamma(t))$, and $(\gamma(0),\dot\gamma(0))\in\tilde{\cal I}(u_0)$. But $\dot\gamma(0)=X(\gamma(0))$, and $\gamma(0)$ can be an arbitrary point of $M$. This implies
$\operatorname{Graph}(X)\subset \tilde{\cal I}(u_0)$. This finishes the proof because we know that $\tilde{\cal I}(u_0)$ is a graph on a part of the base $M$.
\end{proof}

\begin{lemma}
Let $L_X:TM\to\R$ be the Ma\~n\'e Lagrangian associated to the {\rm C}$^k$ vector field $X$ on the compact connected manifold $M$, with $k \geq 2$.
 Assume that $L_X$ satisfies the Mather disconnectedness condition. Then we have:
\item[(1)] The projected Aubry set $\cal A$ is the set of chain-recurrent points of the flow of $X$ on $M$.
\item[(2)] The constants are  the only weak KAM solutions if and only every point of $M$ is chain-recurrent under the flow of  $X$.
\end{lemma}

\begin{proof} To prove (1), we apply Lemma \ref{lemmefond} to obtain that the Aubry set
$\tilde{\cal A}$ is equal to set of points in $\tilde{\cal I}(0)=\operatorname{Graph}(X)$ which are chain-recurrent for the restriction $\phi_t|_{\operatorname{Graph}(X)}$. But from
Proposition \ref{propLagMane} the projection $\pi|_{\operatorname{Graph}(X)}$ conjugates $\phi_t|_{\operatorname{Graph}(X)}$ to the flow of $X$ on $M$. It now suffices to observe that ${\cal A}=\pi(\tilde{\cal A})$.

We now prove (2). Suppose that every point of $M$ is chain-recurrent for the flow of $X$. From what we have just seen ${\cal A}=M$, and so property (1) of Theorem 
\ref{appliDyn} holds. Therefore by property (3) of that same theorem, we have uniqueness up to constants of weak KAM solutions, but the constants are weak KAM solutions.
To prove the converse, assume that the constants are the only weak KAM solutions.
This implies that property (3) of Theorem 
\ref{appliDyn} holds. Therefore by property (4) of that same theorem $\tilde{\cal A}=\tilde{\cal N}$. But $\tilde{\cal I}(0)=\operatorname{Graph}(X)$ is squeezed between $\tilde{\cal A}$ and $\tilde{\cal N}$. Therefore $\tilde{\cal A}=\operatorname{Graph}(X)$. Taking images by the projection $\pi$ we conclude that ${\cal A}=M$. By part (1) of the present lemma,  every point of $M$ is chain-recurrent for the flow of $X$ on $M$.
\end{proof}
Combining this last lemma and Theorem \ref{TheoremePrincipal} completes the proof of  Theorem \ref{TheoremePrincipal2}.

\subsection{Gradient-like vector fields. Examples}
We recall the definition of gradient-like vector filed.

\begin{definition}{\rm
A vector field $X$ on $M$ is said to be gradient-like if we can find a {\rm C}$^1$ function $f:M \to \R$ such that
\begin{enumerate}
\item[(i)] for every $x\in M$, we have $X\cdot f(x)=d_xf(X(x))\leq 0$;
\item[(ii)] for a given $x \in M$, we have $X\cdot f(x)=0$ if and only if $X(x)=0$. 
\end{enumerate}
}
\end{definition}

As an example of gradient-like vector field, we can take $X=-\operatorname{grad}f$, where $f:M \to \R$ is {\rm C}$^1$
and the gradient is taken with respect to the Riemannian metric on $M$. In this case
$$
X\cdot f(x)=-d_xf(\operatorname{grad}f(x))=-\frac12\lVert d_xf \rVert_x^2.
$$
Note that if $\varphi:M\to \R$ is a function such that
$$
\forall x \in M, \quad \varphi(x)=0 \Longleftrightarrow X(x)=0
$$
and $X$ is gradient-like, then $\varphi X$ is also gradient-like.

The following fact is easy to prove.
\begin{proposition}
If $X$ is a {\rm C}$^1$ gradient-like vector field, then the non-wandering set $\O(\phi_t^X)$ is equal to the zero set
$Z(X)=\{x\in M \mid X(x)=0\}$ of $X$ (or equivalently $\O(\phi_t^X)=\operatorname{Fix}(\phi_t^X)$).
\end{proposition}
In the case of Ma\~n\'e's example associated to gradient-like vector field, we have:

\begin{proposition}
Let $X$ be a gradient-like vector field,
and denote by $\Aub$ the Aubry set of the Ma\~n\'e Lagrangian $L_X$.
Then the image of $\Aub^0$ in the quotient Aubry set $(\Aub_M,\d_M)$ is full. 
Therefore, if $X$ is {\rm C}$^k$ with $k \geq 2\dim M -2$, then $\Haus{1}(\Aub_M,\d_M)=0$,
and $L_X$ satisfies the Mather disconnectedness condition.
\end{proposition}
\begin{proof}
If $x\in \Aub$, the whole orbit $\phi_t^X(x)$ is contained in $\Aubry$, and any limit point $x_\i$ of $\phi_t^X(x)$,
as $t\to \i$, is in $\O(\phi_t^X)$, and it is therefore fixed.
We also know by (3) of Theorem \ref{ThMather1} that $\d_M(x,x_\i)=0$.
Therefore the image of $\Aub^0$ in the quotient Aubry set $(\Aub_M,\d_M)$ is full.
The rest of the proof follows by Theorem \ref{THM2}.
\end{proof}

Let us now give some examples.

We start with a Whitney counterexample to Sard Theorem (see for example \cite{Hajl}).
Such a counterexample gives a function $f:\T^n \to \R$ which is C$^{n-1}$, and for which we can find a connected set
$C\subset \T^n$ such that $d_xf=0$ for every $x\in C$, and $f$ is not constant on $C$. Therefore $f(C)=[a,b]\subset\R$, with $a<b$.
If we now consider $X=-\operatorname{grad}f$ and $L_X(v)=\frac12\lVert v-X(x)\rVert_x^2$ on $T\T^n$, then  $f$ is a critical {\rm C}$^1$ subsolution. In fact
$$
H_X(x,d_xf)=d_xf(X(x))+\frac12\lVert d_xf \rVert_x^2=
-\lVert d_xf \rVert_x^2+\frac12\lVert \operatorname{grad}f \rVert_x^2=-\frac12\lVert \operatorname{grad}f \rVert_x^2.
$$
We see that this critical subsolution is strict outside $Z(X)$, therefore we have $Z(X)=\Aub\supset C$.
Since $f$ and $0$ are both critical subsolution, by the proof of Proposition \ref{propDyn2} the function $f$
and so it is $1$-Lipschitz seen as a map from $(\Aub_M,\d_M)$ to $\R$.
This implies that
$$
\Haus{1}(\Aub_M,\d_M) \geq \Haus{1}(f(A)) \geq \Haus{1}(f(C))=\Haus{1}([a,b])=b-a>0.
$$
It follows that, for this $X$, $\Haus{1}(\Aub_M,\d_M)>0$ and $L_X$ does not satisfy the Mather dsconnectedness condition.

Note that we can assume that $f$ is {\rm C}$^\i$ outside $C$. Indeed, if this was not the case, we could approximate $f$ in the C$^{n-1}$ topology
on $M\setminus C$ with a {\rm C}$^\i$ function, so that this approximation glues back with $f$ on $C$ to a C$^{n-1}$ function.

By a standard result (see for example \cite{fathiAMM}), we can find a {\rm C}$^\i$ function $\varphi:M \to [0,+\i[$, with $\varphi|M\setminus C>0$,
$\varphi|C=0$, and such that $\varphi X$ is {\rm C}$^\i$.
Of course the vector field $\varphi X$ is still gradient like, but, since $\varphi X$ is {\rm C}$^\i$,
the associated Ma\~n\'e Lagrangian does satisfy the Mather disconnectedness condition,
and its Aubry set is still $Z(X)$. Note that the orbits of $X$ and $\varphi X$ are the same as $\varphi >0$ on $M\setminus Z(X)$.

We can also modify a little bit $f$ like suggested by Hurley in \cite{Hu} to construct a
C$^{n-1}$ function $f:\T^n \to \R$ such that its Euclidean gradient
$\operatorname{grad}f$ has a chain recurrent point which is not a critical point of $f$, and for which
there exists a connected set $C\subset \T^n$ such that $d_xf=0$ for every $x\in C$, and $f$ is not constant on $C$. 

Although Hurley in \cite[pages 453-454]{Hu} does it for $n=2$ or $3$, starting from a Whitney counterexample to Sard Theorem it
is clear that one can obtain it for any $n\geq 2$.

Note that again, if we take $X=-\operatorname{grad}f$ and we denote by $\Aub_X$ the Aubry set of $L_X$, as above we will have $\Aub_X=Z(X)$,
and in that case the chain recurrent set of $X$ is strictly larger than $\Aub_X$.
Therefore one must have some high differentiability assumption on the vector field $X$ in order to assure that
$\Aub_X$ is equal to the set of chain recurrent points.

Again taking some care in the construction of Hurley, and applying an approximation theorem, we can assume that $f$
is {\rm C}$^\i$ outside $C$. Like above we can find a {\rm C}$^\i$ function $\varphi:\T^n \to [0,+\i[$, with $\varphi|\T^n\setminus C>0$, $\varphi|C=0$
and such that $\varphi X$ is {\rm C}$^\i$.
Note that $\Aub_{\varphi X}$ is equal to the chain recurrent set of $\varphi X$ (which is the same as the chain recurrent set of $X$)
because $L_{\varphi X}$ satisfies the Mather disconnectedness condition.

\appendix
\section{A Lemma of Ferry and a result of Mather}
\subsection{Ferry's Lemma}

In this appendix, we state and prove a generalization of a Lemma due to Ferry in 1976 \cite{ferry76}. This Lemma has been rediscovered Bates in 1992 \cite{bates93} to prove his generalization of Sard's Theorem. They
proved that 
if $E \subset \R^n$ is a measurable set, $f:E \rightarrow \R$ is continuous, and $n \geq 2$ is such that $f$ satisfies
$$
 \forall x,y \in E,\quad |f(x)-f(y)| \leq C \|x-y\|^n,
$$
then $f(E)$ has Lebesgue measure $0$. 

Their proof yield in fact the following generalization.
\begin{lemma}
 \label{lemNEW1}
Let $\Psi:E\rightarrow X$ be a map where $E$ is a subset of $\R^n$ and $(X,d_X)$ is a semi-metric space. Suppose that there are $p$ and $M$ such that 
$$
\forall x,y \in E, \quad d_X (\Psi(x),\Psi(y)) \leq M \|x-y\|^{p}.
$$
If $p>1$, then the $n/p$-dimensional Hausdorff measure of $(\Psi(E),d_X)$ is $0$.
\end{lemma}

\begin{proof} Since all norms on $\R^n$ are equivalent we can assume
$$\forall x=(x_1,\dots,x_n)\in\R^n, \quad \|x\|=\max_{i=1}^n|x_i|.$$
Since it suffices to prove that $\Haus{\frac np}(\Psi(E\cap K))=0$ for each compact set $K \subset \R^n$, we can assume  that $E$ is bounded, which in particular implies $\Leb{n}(E) < +\i$ (we denote by $\Leb{n}$ the Lebesgue measure on $\R^n$). We now write $E=E_1 \cup E_2$, where $E_1$ is the set of density points for $E$ and $E_2=E \setminus E_1$. By the definition if density points
$$\forall x\in E_1, \quad\lim_{r\to0}\frac{\Leb{n}(E_1 \cap B(x,r))}{\Leb{n}(B(x,r))}=1.$$
 It is a standard result in measure theory that $\Leb{n}(E_2)=0$. Thus for each $\epsilon >0$ be fixed, there exists a countable family of balls $\{B_i\}_{i \in I}$ such that
$$
E_2 \subset \bigcup_{i \in I} B_i \quad \mbox{and} \quad \sum_{i \in I} (\operatorname{diam}B_i)^n \leq \e.
$$
Then we have
\begin{multline*}
\Haus{\frac np}(\Psi(E_2)) \leq \sum_{i \in I} \left( \operatorname{diam}_X \Psi(B_i\cap E_2) \right)^{\frac np}\\
\leq M \sum_{i \in I} [(\operatorname{diam}B_i)^p]^{\frac np} \leq M \sum_{i \in I} (\operatorname{diam}B_i)^n \leq M\e.
\end{multline*}
Letting $\e \rightarrow 0$, we obtain $\Haus{^{\frac np}}(\Psi(E_2))=0$. Note that in this part of the argument we have not used the condition $p>1$\\
We now want to prove that $\Haus{^{\frac np}}(\Psi(E_1))=0$.
Fix $N \in \N$. For every density point $x \in E_1$, there exists  $\rho(x)>0$ such that
$$
\forall r \leq \rho(x), \quad \frac{\Leb{n}(E_1 \cap B(x,r))}{\Leb{n}(B(x,r))}=\frac{\Leb{n}(E \cap B(x,r))}{\Leb{n}(B(x,r))}\geq 1-\frac{1}{2N^n}. 
$$
Note that, since $\Leb{n}(B(y,s))=2^ns^n$, this implies that for such an $x\in E_1$, we have
$$\forall r\leq \rho(x),\, \forall y\in\R^n,\quad {\Leb{n}(B(x,r)\setminus E_1  )\leq \frac12 \Leb{n}(B(y,r/N))},$$
Therefore, since for $y\in B(x,\frac{N-1}Nr)$, we have $B(y,r/N))\subset B(x,r)$, we obtain
\begin{equation}
\label{eqlemmanorton}
\forall r\leq \rho(x),\, \forall y\in B(x,\frac{N-1}Nr),\quad E_1\cap  B(y,r/N)\neq\emptyset.
\end{equation}
Fix $x\in E_1$. It is now simple to prove that for all $y \in E_1 \cap B(x,r)$, with $r\leq \rho(x)$, there exist $N+1$ points
$x_0, \ldots, x_N \in E_1$, with $x_0=x$ and $x_N=y$ such that
$$
\forall 1  \leq i \leq N,\quad
|x_i - x_{i-1}| \leq \frac{3r}{N}.
$$
Indeed, first take $y_1,\ldots,y_{N-1}$ the $N-1$ points on the line segment $[y,x]$ such that $|y_i - y_{i-1}|=\frac{|y-x|}{N}.$
We then observe that, for $i=1,\dots,N-1$, we have $\|y_i-x\|\leq i{|y-x|}/{N}\leq  (N-1)r/N$. Hence, by (\ref{eqlemmanorton}), the intersection
$B(y_i,\frac{r_x}{N}) \cap E_1$ is not empty for each $i=1,\dots,N-1$, and so
it suffices to take a point $x_i$ in that intersection.
Then, for all $y \in E_1 \cap B(x,r)$,
\begin{multline}
d_X(\Psi(x),\Psi(y)) \leq \sum_{i=1}^N d_X(\Psi(x_{i-1}),\Psi(x_{i})) \leq M \sum_{i=1}^N |x_i - x_{i-1}|^p \\
\leq M N \left(\frac{3r}{N}\right)^p =3^{p} M N^{1-p} r^{p}.
\end{multline}
It follows that 
\begin{equation}
\label{ineqhausdmeas2}
\begin{split}
\forall x\in E_1,\,\forall r\leq \rho(x), \quad \operatorname{diam}(\Psi(B(x,r)\cap E_1)&\leq 2\left(3^{p} M N^{1-p} r^{p}\right)\\
&=2^{1-p}3^pM N^{1-p} [\operatorname{diam}(B(x,r))]^{p}
\end{split}
\end{equation}
We are now able to prove that $\Haus{{\frac np}}(\Psi(E_1))=0$.\\
Take an open set $\O \supset E_1$ such that
$\Leb{n}(\O) \leq \Leb{n}(E_1)+1=\Leb{n}(E) +1<+\i$, and consider the fine covering $\F$ given by $\F=\{B(x,r)\}_{x \in E _1}$ with $r$
such that $B(x,r) \subset \O$ and $r \leq \frac{\rho(x)}{5}$, where $\rho(x)$ was defined above.
By Vitali's covering theorem (see \cite[paragraph $1.5.1$]{EvaGar}), there
exists a countable collection $\GG$ of disjoint balls in $\F$ such that
$$
E_1 \subset \bigcup_{B \in \GG} 5B,
$$
where $5B$ denotes the ball concentric to $B$ with radius $5$ times that of $B$.
Since the balls in $\F$ are disjoint and contained in $0$,
we get
$$\sum_{B\in \GG}\Leb{n}(B)\leq \Leb{n}(O)\leq \Leb{n}(E) +1<+\i.$$
Since the norm on $\R^n$ is the max norm, we have $\Leb{n}(B)=\operatorname{diam}(B)^n$ for every $B$ which is a ball for the norm. 
Therefore
\begin{equation}\label{ineqrayons}
\sum_{B\in \GG}\operatorname{diam}(B)^n\leq \Leb{n}(O)\leq \Leb{n}(E) +1<+\i.
\end{equation}
We can so consider the covering of $\Psi(E_1)$ given by $\cup_{B \in \GG} \Psi(5B \cap E_1)$.
In this way, by (\ref{ineqhausdmeas2}), we get
\begin{align*}
\Haus{{\frac np}}(\Psi(E_1)) &\leq \sum_{B \in \GG} \left( \operatorname{diam}_X \Psi(5B \cap E_1) \right)^{{\frac np}}\\
&\leq \sum_{B \in \GG} \left( 2^{1-p}3^pM N^{1-p} [5\operatorname{diam}(B)]^{p} \right)^{{\frac np}}\\
&= \sum_{B \in \GG}2^{\frac{n(1-p)}p}3^nM^{\frac{n}p} N^{\frac{n(1-p)}p}5^n\operatorname{diam}(B)]^{n} \\
&=2^{\frac{n(1-p)}p}3^nM^{\frac{n}p} N^{\frac{n(1-p)}p}5^n \sum_{B \in \GG}\operatorname{diam}(B)]^{n}.
\end{align*}
Using (\ref{ineqrayons}), we obtain
$$\Haus{\frac np}(\Psi(E_1))\leq  2^{\frac{n(1-p)}p}3^nM^{\frac{n}p} N^{\frac{n(1-p)}p}5^n\left(\Leb{n}(E) +1\right).$$
Since $\Leb{n}(E) +1<\i$ and $1-p<0$, letting $N\to\i$ we obtain $\Haus{{\frac np}}(\Psi(E_1))=0$.
\end{proof}
\begin{remark}\label{BatesVersusFerry}{\rm As we said at the beginning of the appendix the original case of Ferry's Lemma plays a crucial role in
Steve Bates \cite{bates93} version of Morse-Sard Theorem: If $f:M\to \R$ is of class {\rm C}$^{n-1,1}$, where $n=\dim M\geq2$, then the set of critical values of $f$ is of Lebesgue measure $0$. 

In fact the original case of Ferry's Lemma is also a consequence of Bates
\cite{bates93} version of Morse-Sard Theorem. Indeed note first that, by uniform continuity,
we can extend $f$ to the closure $\bar E$ of $E$ in $\R^n$. Of course by continuity
we will also have
$$
 \forall x,y \in \bar E, \quad |f(x)-f(y)| \leq C \|x-y\|^n.
$$
On the closed set the family $f,Df=0,\dots, D^{n-1}f=0$ satisfy the condition of Whitney's
extension theorem with $D^{k-1}f$ Lipschitz (see \cite[Theorem 4, page 177]{stein}), therefore the exists an extension $\bar f:\R^n\to \R$ which is of class {\rm C}$^{n-1,1}$. Of course all points of $\bar E$ are critical points of $\bar f$ so by Bates version of the Morse-Sard theorem
$\bar f(\bar E)=f(\bar E)$ has measure $0$.}
\end{remark}
It is easy to generalize this result to a finite dimensional manifolds, since such manifolds are always assumed metric and separable, and therefore second countable.

Before stating this generalization, we recall that on a smooth (in fact at least {\rm C}$^1$) finite dimensional manifold $M$ the notion of locally H\"older of exponent $p\geq 0$ makes sense. A map $f:A\to X$ where $(X,d_X)$ is a metric space and $A\subset M$ is said to be locally H\"older of exponent $p$ (we allow $p\geq 1$!) if for every $x\in A$, we can find a neighborhood $U_x$ of $x$ and $M_x<\i$ such that
$$\forall y,y'\in U_x\cap A,\quad  d_X(f(y),f(y'))\leq M_xd_M(y,y')^p,$$
where $d_M$ is a distance obtained from a Riemannian metric on $M$. Note that this notion is independent of the choice of $d_M$, since all distances obtained from
Riemannian metrics are locally Lipschitz equivalent. It is not difficult to show that
$f:A\to X$ is locally H\"older of exponent $p$ if and only if we can find a family 
$(U_i,\varphi_i)_{i\in I}$ of smooth (or at least {\rm C}$^1$) charts of $M$, with 
$U_i$ open subset of $\R^n$, where $n=\dim M$, and a family $M_i\in I$ of finite numbers such that $A\subset \cup_{i\in I} \varphi_i(U_i)$ and
$$\forall i\in I,\forall x,x'\in U_i,\quad d_X(f\varphi_i(x),f\varphi_i(x'))\leq M_i\|x-x'\|^p,$$
where $\|\cdot\|$ is a norm on $\R^n$. Since $M$ is second countable we can always assume that $I$ is itself countable, and therefore we can deduce the following generalization of Lemma \ref{lemNEW1}.
\begin{lemma}\label{lemNEW2}%
{Let $M$ be a (metric separable) manifold of dimension $n<\i$, and $(X,d_x)$ be a metric space. Suppose $\Psi:A\to X$, where $A\subset M$, is a locally
H\"older map of exponent $p>1$. Then the $n/p$-dimensional Hausdorff measure of $(\Psi(A),d_x)$ is $0$.}
\end{lemma}

\subsection{Mather's result}

We would like to show how one can deduce from Ferry's Lemma the following
result of Mather, compare with \cite[Proposition 1 page 1507]{mather02}
\begin{proposition}\label{mather02Proposition1}
{Let $X$ be a compact, connected subset of $\R^d , d \geq 2$. Let $x, y \in X$
and   $\epsilon > 0$. Then there exists a sequence $x = x_0, \dots ,x_k = y$ of points in $X$ such that $\sum_{i=0}^{k-1}\| x_{i+1} - x_i \|^d < \epsilon$. }
\end{proposition}
In fact, if $(A,d)$ is a metric space and $p>0$, we can introduce a semi-metric $\delta_p$ on $A$ defined by
$$\delta_p(a,a')=\inf\{\sum_{i=0}^{k-1}d(a_{i+1},a_i)^p\mid k\geq 1, a_1,\dots a_{k-1}\in A, a_0=a,a_k=a'\}.$$
It is not difficult to check that $\delta_p$ is symmetric, satisfies the triangular inequality, 
and that $\delta_p(a,a)=0$ for every $a\in A$. Note that when $p\leq 1$, the function $d^p$ is already a metric. Therefore it follows by the triangular inequality that 
$\delta_p=d^p$, when $p\leq 1$. However when $p>1$, we might have $\delta_p(a,a')=0$ with $a\neq a'$. This is indeed the case when $A=[0,1]$ with distance $d(t-t')=|t-t'|$. In fact, if we divide the segment $[t,t']$ by $N$ equally space points, we obtain
$\delta_p(t,t')\leq N(|t-t'|/N)^p$, hence, letting $N\to \i$, since $p>1$ we obtain 
$\delta_p=0$.
This yields the first of the following remarks.
\begin{remark}\label{remarquesurdeltap}{\rm

1) If $p>1$ and there exists a Lipschitz curve $\gamma:[0,1]\to A$,
with $\gamma(0)=a$ and  $\gamma(1)=a'$, then $\delta(a,a')=0$, for every $p>1$.

2) We will say that  $A$ is Lipschitz arcwise connected, if for every $a,a'\in A$ there is a 
Lipschitz curve $\gamma:[0,1]\to A$,
with $\gamma(0)=a$ and  $\gamma(1)=a'$. It follows from 1) that $\hat \delta_p\equiv 0$ if $A$  is Lipschitz arcwise connected and $p>1$.

3) If $M$ is a connected smooth manifold with a distance $d$ coming from a Riemannian metric, then $\delta_p\equiv 0$, for every $p>1$. This follows from 1) above since any two points in a connected manifold can be joined by a smooth path.

4) If $A'\subset A$ we can consider the distance $\delta'_p$ associated to $(A',d|A)$ and $p>0$. We always have $\delta_p|A' \leq \delta'_p$ with equality when $A'$ is dense in $A$.}

5) If $f:A\to B$ is Lipschitz with constant Lipschitz constant $\leq K$, then $f$ is also Lipschitz as a map from 
$(A,\delta^A_p)$ to $(B,\delta^B_p)$, with constant Lipschitz constant $\leq K^p$
\end{remark}

In the sequel, we will denote by $(\hat A_p,\delta_p)$, or just by $\hat A_p$, the metric space obtained by identifying points $a,a'\in A$ such that $\delta_p(a,a')=0$. We denote by $\hat \pi_p:
A\to \hat A_p$ the canonical projection. It is clear that $\delta_p(a,a')\leq d(a,a')^p$, therefore the projection is H\" older of exponent $p>0$. It follows that one has the following consequence of Lemma \ref{lemNEW2}.
\begin{proposition}\label{1mather02generalized}{Suppose that $A$ is a subset of an $n$-dimensional manifold $M$, and that $d$ is a distance that is locally Lipschitz equivalent to a the restriction to $A$ of a distance on $M$ coming from a Riemannian metric. Then $\Haus{\frac np}(\hat A_p)=0$, for all $p>1$.
In particular, if $n\geq 2$, we have $\Haus{1}(\hat A_n)=0$, and therefore $\hat A_n$
is totally disconnected.}
\end{proposition}
This proposition follows from Lemma \ref{lemNEW2}, except for the last statement 
which is a general fact: If a metric space $X$ has $0$ $1$-dimensional Hausdorff 
measure, it is totally disconnected. In fact, if $x$ is fixed, note that the map 
$d_x:X\to\R, y\mapsto d(x,y)$ is Lipschitz; hence the image $d_x(X)$ has also 
1-dimensional Hausdorff measure, i.e.\ Lebesgue measure, in $\R$ equal to $0$. In particular, we can find a sequence $r_n>0$, with $r_n\to 0$, and $r_n\notin d_x(X)$. This last  condition means that$\{y \in X\mid d(x,y)=r_n\}$ is empty, therefore the boundary  of the ball $\bar B_d(x,r_n)$ is empty.

It is now easy to obtain Proposition \ref{mather02Proposition1}. In fact,
if under the hypotheses of Proposition \ref{1mather02generalized} we also assume that
$A$ is connected, then $\hat A_n$ is also connected because $\hat\pi_p$ is continuous and surjective. But a connected and totally disconnected metric space contains at most one point, therefore $\delta_n(x,y)=0$ for every pair of points in the connected subset $A$ of $\R^n$, when $n\geq 2$.

Note that we could have obtained Proposition  \ref{mather02Proposition1} directly from Bates \cite{bates93} version of Morse-Sard Theorem along the lines mentioned in Remark \ref{BatesVersusFerry}.

Mather gave an extension Proposition  \ref{mather02Proposition1}  to Lipschitz laminations, see \cite[Proposition 2, page 1510]{mather02}. In fact, by our method we can give a much more general result. For this we introduce the following definition.
\begin{definition}[Agglutination]{\rm A subset $A$ of the finite $n$-dimensional manifold $M$ is a Lipschitz agglutination of codimension $k$, if every $x\in A$ is contained in a 
subset $B\subset A$ which is Lipschitz arcwise connected and of topological dimension
$\geq n-k$.}
\end{definition}
Obviously any subset of the manifold which admits a codimension $k$ Lipschitz lamination, as considered in \cite{mather02}, is a codimension $k$  Lipschitz agglutination. Moreover, any union of  Lipschitz agglutination of codimension $k$ is itself  a Lipschitz agglutination of codimension $k$. In particular any union of codimension $k$
immersed Lipschitz submanifolds is a  Lipschitz agglutination of codimension $k$.
We can now state our generalization.
\begin{proposition}\label{2mather02generalized}�{Suppose that $A$ is a codimension $k$ Lipschitz  agglutination of the $n$-dimensional manifold $M$, and that $d$ is a distance that is locally Lipschitz equivalent to a the restriction to $A$ of a distance on $M$ coming from a Riemannian metric. Then $\Haus{\frac kp}(\hat A_p)=0$ for all $p>1$.
In particular, if $k\geq 2$, we have $\Haus{1}(\hat A_k)=0$, and therefore $\hat A_k$
is totally disconnected.}
\end{proposition}
We first prove a well-known Lemma.
\begin{lemma}{If $M$ is a finite-dimensional (metric separable) manifold, and $d$ is an integer with $0\leq d\leq n$, we can find a
sequence $(D_i)_{i\in \N}$ of subsets of $M$, each of which is  {\rm C}$^\infty$ diffeomorphic to a Euclidean disc of dimension $n-d$, such that
the topological dimension of $M\setminus\cup_{i\in N} D_i$ is  $\leq d-1$.
In particular, any subset $B$ of $M$ of topological dimension $\geq d$ has to intersect one of the $D_i$.}
\end{lemma} 
\begin{proof} We first consider the case $M=\R^n$. Call $S_n^d$ the family of subsets of
$\{1,\dots,n\}$ with exactly $d$ elements. For every $I\in S_n^d$ and every $(r_1,\dots,r_n)\in\Q^n$, we define
$$V^{I}_{(r_1,\dots,r_n)}=\{(x_1,\dots,x_n)\in\R^n\mid x_j=r_j, \forall j\in I\}.$$
Each $V^{I}_{(r_1,\dots,r_n)}$ is an affine subspace of dimension $n-d$,
and this family is countable.

If we denote by ${\cal M}^{d-1}_n$ the complement in $\R^n$ of the countable 
union of the subsets $V^{I}_{r},I\in S_n^d, r\in\Q^n$, then the points in
${\cal M}^{d-1}_n$ are precisely the points in $\R^n$, who have at most $d-1$ rational coordinates. By \cite[Example III.6, page 29]{HurWall} the topological dimension of ${\cal M}^{d-1}_n$ is $\leq d-1$ (in fact it is $d-1$).

We now consider a general (metric separable) $n$-dimensional smooth manifold $M$. 
We can find a countable family of  charts $\varphi_j:\R^n\to M,j\in \N$ such that 
$\cup_{j\in\N}\varphi_j(\bar\BB)=M$, where $\bar\BB$ is the unit closed Euclidean ball in
$\R^n$. We consider the countable collection $D_{j,I,r}, j\in \N,I\in S_n^d,r\in \Q^n$ defined by
$$D_{j,I,r}=\varphi_j(V^I_r).$$
Each  $D_{j,I,r}$ is {\rm C}$^\infty$ diffeomorphic to a Euclidean disc of dimension $n-d$.
We now show that the topological dimension of the complement
$$\mathscr{C}=M\setminus\bigcup_{j\in \N,I\in S_n^d,r\in \Q^n}D_{j,I,r},$$
is $\leq d-1$. We can write $\mathscr{C}=
\cup_{j\in\N} \mathscr{C}\cap \varphi_j(\bar\BB)$.
Since each $ \mathscr{C}\cap \varphi_j(\bar\BB)$ is closed in $ \mathscr{C}$, by the Countable Sum Theorem \cite[Theorem III.2, page 30]{HurWall}, it suffices to show that
each $ \mathscr{C}\cap \varphi_j(\bar\BB)$  has topological dimension $\leq d-1$.
But, the map $\varphi_j^{-1}:  \varphi_j(\R^n)\to \R^n$  sends $ \mathscr{C}\cap \varphi_j(\bar\BB)$ to a subset of ${\cal M}^{d-1}_n$ has topological dimension $\leq d-1$. This implies that the topological dimension of $ \mathscr{C}\cap \varphi_j(\bar\BB)$ is
$d-1$ by \cite[Theorem III.1, page 26]{HurWall}. Note that this last reference proves also the last statement in the Lemma.
\end{proof}

\begin{proof}[Proof of Proposition \ref{2mather02generalized}] We apply the lemma above with $d=n-k$, to obtain a countable family $D_i,i\in \N$ of {\rm C}$^\infty$ discs of dimension
$n-d=k$ such that each subset of $M$ whose topological dimension is $\geq d=n-k$ has to intersect one of the $D_i$. Consider then a Lipschitz agglutination $A\subset M$ of codimension $k$, and we fix $p>1$. We first claim that $\hat A_p=\cup_{i\in \N}
\hat\pi_p(A\cap D_i)$. In fact if $x\in A$, by the definition of a Lipschitz agglutination of codimension $k$ we can find a Lipschitz arcwise connected subset $B_x\subset A$ of dimension $\geq n-k$ containing $x$. By the property of the family $D_i$, there exists $i_0\in I$ such that  $B_x\cap D_{i_0}\neq\emptyset$. Choose $y\in B_x\cap D_{i_0}$.
By 2) of Remark \ref{remarquesurdeltap}, we have $\delta_p^{B_x}(x,y)=0$. Since $B_x\subset A$, we conclude that $\delta_p^{A}(x,y)=0$. Therefore $\hat \pi_p(x)=\hat \pi_p(y)\in \hat\pi_p(B_x\cap D_i)\subset \hat\pi_p(A\cap D_i)$. Since the family $D_i$ is countable, it remains to show that $\Haus{k/p}(D_i\cap A, \delta_p^A)=0$.
Note that since $D_i$ is a submanifold of $M$, the distance $d$ on $M$ induces a distance on $D_i$ which is locally Lipschitz equivalent to a distance coming from a Riemannian metric. Therefore  by Proposition \ref{1mather02generalized}, we have
$\Haus{k/p}(D_i\cap A, \delta_p^{D_i\cap A})=0$. But the inclusion $D_i\cap A\hookrightarrow A$ is Lipschitz with Lipschitz constant $1$ for the metrics $\delta_p^{D_i\cap A}$ on $D_i\cap A$ and $\delta_p^{ A}$ on $A$. Therefore $\Haus{k/p}(D_i\cap A, \delta_p^A)=0$.
\end{proof}

\section{Existence of {\rm C}$^{1,1}_{loc}$ critical subsolution on noncompact
manifolds}\label{BernardNonCompact}

In \cite{bernard}, using some kind of Lasry-Lions regularization (see
\cite{ll86}), Bernard proved the existence of {\rm C}$^{1,1}$ critical
subsolutions on compact manifolds. Here, adapting his proof, we show that
the same result holds in the noncompact case and we make clear that the
Lipschitz constant of the derivative of the {\rm C}$^{1,1}_{loc}$ critical
subsolution can be uniformly bounded on compact subsets of $M$.
We consider the two Lax-Oleinik semi-groups $T_t^-$
and $T_t^+$ defined by,
\begin{align*}
T_t^-u(x)&=\inf_{y \in M} u(y)+h_t(y,x)\\
T_t^+u(x)&=\sup_{y \in M} u(y)-h_t(x,y),
\end{align*}
for every $x\in M$. 

For any $c\in M$, these two semi-groups preserve the set of functions dominated by 
$L+c$, see for example  \cite{fathibook} for the compact case or \cite{fm}  for the non-compact case. It is also well known that these semi-groups have some regularizing effects: namely 
for every $t>0$ and every Lipschitz (or even continuous, when $M$ is compact) function $u:M\to\R $, the function
$T_t^+u$ is finite everywhere and  locally semi-convex, while $T_t^-u$ is finite everywhere and locally semi-concave, see for example \cite{fathibook} or the explanations below.

In \cite{bernard}, the idea for proving the existence of {\rm C}$^{1,1}$ critical
subsolution on compact manifolds is the following: 
it is a known fact that a function is {\rm C}$^{1,1}$ if and only if it is both locally
semi-concave and locally semi-convex.
Let now $u$ be a critical viscosity subsolution. If we apply the semi-group
$T_t^+$ to $u$,
we obtain a semi-convex critical viscosity subsolution $T_t^+u$.
Thus, if one proves that, for $s$ small enough, $T_s^-T_t^+u$ is still
semi-convex, as
we already know that it is semi-concave, we would have found a {\rm C}$^{1,1}$
critical subsolution. Since we want to give a uniform bound on the
Lipschitz constant of the derivative of the {\rm C}$^{1,1}_{loc}$ critical
subsolution on compact sets, we will have to bound the constant of
semi-convexity of $T_t^+u$ on compact subsets of $M$. Let us now prove the
result in the noncompact case. \\

\begin{theorem}
\label{THMexistC11}
Assume that $H$ is of class {\rm C}$^2$. For every compact subset $K$ of
$M$, there is a constant
$\ell=\ell(K)>0$ such that, if $u :M\rightarrow \R$ is a critical
viscosity subsolution, then there exists a {\rm C}$^{1,1}_{loc}$ critical
subsolution $v:M \rightarrow \R$ whose restriction to the projected Aubry
set is equal to $u$ and such that the mapping $x \mapsto (x,d_x v)$ is
$\ell $-Lipschitz on $K$.
\end{theorem}

Before proving Theorem \ref{THMexistC11}, we need a few lemmas.
\begin{lemma}
 \label{globalLip}
There is a constant $A<+\infty$ such that for any $c\in\R$, any function 
$u: M \rightarrow \R$ dominated by $c$ is $(A+c)$-Lipschitz on $M$, that is
$$
 \forall x,y \in M, \quad \lvert u(y)-u(x) \rvert \leq (A+c)d(x,y),
$$
where $d$ denotes the Riemannian distance associated to the Riemannian metric $g$ on $M$.
\end{lemma}
\begin{proof}
Let $u:M \rightarrow \R$ be dominated by $L+c$ and $x,y \in
M$ be fixed. Let $\g_{x,y}:[0,d(x,y)] \rightarrow M$ be a minimizing
geodesic with constant unit speed joining $x$ to $y$. By definition of
$h_{d(x,y)}(x,y)$, one has
$$
h_{d(x,y)}(x,y) \leq \int_0^{d(x,y)} L(\g_{x,y}(t),\dot \g_{x,y}(t)) \,dt
\leq Ad(x,y),
$$
where $A=\sup_{x \in M} \{L(x,v) \mid \|v\|_x  \leq 1\}$ is finite
thanks to the uniform boudedness of $L$ in the fibers. Thus, one has
$$
u(x)-u(y) \leq h_{d(x,y)}(x,y) + cd(x,y) \leq (A(1)+c)d(x,y).
$$
Exchanging $x$ and $y$, we conclude that $u$ is $(A+c)$-Lipschitz.
\end{proof}
Next we give some estimates on the functions $h_t$.
\begin{lemma}\label{Estimatesht}{There exists a constant   $B<+\infty$ such that
$$\forall t>0,\,\forall x\in M, \quad h_t(x,x)\leq Bt.$$
Moreover, for every constant $C<+\infty$, we can find $D(C)>-\infty$ such that
$$\forall t>0,\,\forall x,y\in M,\quad h_t(x,y)\geq Cd(x,y)+D(C)t.$$}
\end{lemma}
\begin{proof} Using a constant curve at $x$, we get 
$$h_t(x,x)\leq \int_0^tL(x,0)\,ds.$$
Therefore, if we set $B=\sup\{L(x,0)\mid x\in M\}<+\infty$, we obtain
$$\forall t>0,\, \forall x\in M,\quad h_t(x,x)\leq Bt.$$
Using the uniform superlinearity of $L$, for every $C<+\infty$ we can find a constant 
$D(C)>-\infty$, depending only on $C$,  such that
$$\forall (x,v)\in TM,\quad L(x,v)\geq C\lVert v\rVert_x+D(C).$$
Fix now $x,y\in M$. If $\gamma:[0,t]\to M$ is such that $\gamma(0)=x$, $\gamma(t)=y$, we can apply the above equality to $(\gamma(s),\dot\gamma(s))$ and integrate to obtain
$$
\int_0^t L(\gamma(s),\dot\gamma(s))\,ds\geq C \operatorname{length}(\gamma)+D(C)t
\geq Cd(x,y)+ D(C)t.
$$
To find $h_t(x,y)$, we have to minimize $ \int_0^t L(\gamma(s),\dot\gamma(s))\,ds$ over all curves with $\gamma(0)=x$, $\gamma(t)=y$. Therefore, by what we just obtained, we get
$$ h_t(x,y)\geq Cd(x,y)+D(C)t.$$
\end{proof}
\begin{lemma}\label{ConvToZero}{If $C<+\infty$ is a given constant, we can find $B(C)<+\infty$ such that
for every $u:M\to\R$ which is Lipschitz, with Lipschitz constant $\leq C$, we have
\begin{align*}
\forall t\geq 0,\, \forall x\in M,\quad & T^-_tu(x)=\inf\{u(y)+h_t(y,x)\mid y\in M, d(x,y)\leq B(C)t\},\\
& T^+_tu(x)=\sup\{u(y)-h_t(x,y)\mid y\in M, d(x,y)\leq B(C)t\},\\
& \lvert T^-_tu(x)-u(x)\rvert \leq B(C)t,\\
& \lvert T^+_tu(x)-u(x)\rvert \leq B(C)t.
\end{align*}}
\end{lemma}
\begin{proof} We will do the proof for $T^-_t$, as the case of $T^+_t$ is analogous.
Using the first part of Lemma \ref{Estimatesht}, we get 
$$T_t^-u(x)\leq u(x)+h_t(x,x)\leq u(x)+Bt$$
By the second part of Lemma \ref{Estimatesht}, we get
$$T^-_tu(x)\geq \inf_{y\in M}u(y)+Cd(x,y)+D(C)t.$$
Since $u$ is $C$-Lipschitz, we have $u(x)\leq u(y)+Cd(x,y)$,
hence $T^-_tu(x)\geq u(x)+D(c)t$. It follows that 
$$\lvert T^-_tu(x)-u(x)\rvert \leq \max\{B,-D(C)\}t.$$
Since $u(x)+h_t(x,x)\leq u(x)+Bt$,
we obtain
$$T^-_tu(x)=\inf\{u(y)+h_t(y,x)\mid y\in M, u(y)+h_t(y,x)\leq u(x)+Bt\}.$$
Using again the second part of Lemma \ref{Estimatesht}, and the fact that $u$ is $C$-Lipschitz, we know that 
\begin{align*}
u(y)+h_t(y,x)&\geq u(y) + (C+1)d(x,y)+D(C+1)t\\
&\geq u(x)+d(x,y)+D(C+1)t.
\end{align*}
It follows that
$$T^-_tu(x)=\inf\{u(y)+h_t(y,x)\mid y\in M, d(x,y)\leq Bt-D(C+1)t\}.$$
Hence we can take as $B(C)$ any finite number $\geq\max\{B, -D(C), B-D(C+1)\}$.
\end{proof}
For the next lemmas we need to introduce some notation.
We will suppose that $(U,\varphi)$ is a {\rm C}$^\infty$ chart on $M$. Here $U$ is an open subset, and $\varphi:U\to \R^k$ is a {\rm C}$^\infty$ diffeomorphism on the open subset $\varphi(U)$ of $\R^k$. We will denote by $\lVert \cdot\rVert_{\mathrm {euc}}$ the canonical Euclidean norm on $\R^k$. For $r\geq 0$, we set
 $$\BB(r)=\{v\in\R^k\mid \lVert v\rVert_{\mathrm {euc}}\leq r\},$$
i.e. the subset   $\BB(r)$ is the closed Euclidean ball of radius $r$ and center $0$ in $\R^k$. 
\begin{lemma}\label{C11Conservee}{Suppose that $(U,\varphi)$ is a {\rm C}$^\infty$ chart on $M$, and $\BB(r)\subset \varphi (U)$. For any $r'<r$, any $A\geq 1$, any $B\geq1$, and any $\epsilon>0$, we can find $\delta>0$ such that for any function $u:\BB(r)\to\R$ satisfying
\begin{itemize}
\item[(a)] the function $u$ is {\rm C}$^{1,1}_{loc}$
 on $\BB(r)$;
\item[(b)] the Lipschitz constant (for the canonical Euclidean metric on $\R^k$) of $u$ on $\BB(r)$ is  $\leq A$;
\item[(c)] the Lipschitz constant (for the canonical Euclidean metric on $\R^k$) of the derivative $x\mapsto d_x(u\circ\varphi^{-1})$ on $\BB(r)$ is bounded by $B$;
\end{itemize}
and any $t\leq \delta$, the function $T^{-,\varphi}_tu:\BB(r)\to \R$ defined by
$$T^{-,\varphi}_tu(x)=\inf_{y\in \BB(r)}u(y)+h_t(\varphi^{-1}(y),\varphi^{-1}(x))$$
satisfies
\begin{itemize}
\item[\rm (a')] the function $T^{-,\varphi}_tu$  is {\rm C}$^{1,1}$
on a neighborhood of $\BB(r')$;
\item[\rm (b')] the Lipschitz constant (for the canonical Euclidean metric on $\R^k$) of $T^{-,\varphi}_tu$ is bounded by $A+\epsilon$;
\item[\rm (c')]  the Lipschitz constant  (for the canonical Euclidean metric on $\R^k$) of $x\mapsto d_x(T^{-,\varphi}_tu)$ on $\BB(r')$ is bounded by $B+\epsilon$;
\item[\rm (d')] for every $x\in \BB(r')$, there is one and only one $y_x\in \BB(r)$ such that
$$\varphi^{-1}(x)=\pi^*\phi_t^H(\varphi^{-1}(y_x),d_{\varphi^{-1}(y_x)}(u\circ\varphi)),$$
where $\pi^*:T^*M\to M$ is the canonical projection, and $\phi^H_t$ is the Hamiltonian flow of $H$ on $T^*M$. Moreover, we have
$$ (\varphi^{-1}(x),d_{\varphi^{-1}(x)}(u\circ\varphi))=
\phi_t^H(\varphi^{-1}(y_x),d_{\varphi^{-1}(y_x)}(u\circ\varphi)).$$
\end{itemize}}
\end{lemma}
\begin{proof} We can assume that $r<+\infty$. To simplify notations, we will suppose that $\varphi$ as the ``identity'', i.e. we will write things in the coordinate system given by $\varphi$. 
Let us choose $r''$ and $R$ such that $r'<r''<r<R$ and $\varphi(U)\supset \BB(R)$.
If we set $A_1=\sup\{L(x,v)\mid x\in\BB(R), \lVert v\rVert_{\mathrm {euc}}\leq 1\}$,
any function $u\prec L+c$ has, on $\BB(R)$, a Lipschitz constant $\leq A=A_1+c$.
In particular $\lVert d_xu\rVert_{\mathrm {euc}}\leq A$ at every point $x\in \BB(R)$ where 
$d_xu$ exists.

By continuity and compactness we can find $\delta_1>0$ such that
$$\forall x\in \BB(r),\,\forall p\in (\R^k)^*\text{ with } \lVert p\rVert_{\mathrm {euc}}\leq A ,\,\forall t\in[- \delta_1,\delta_1],\quad
\phi_t^H(x,p)\in \BB(R)\times (\R^k)^*.$$

By Lemma \ref{ConvToZero} we can find $\delta_2>0$, with $\delta_2\leq \delta_1$ depending only $c$, 
such that for any function $u:M\to \R$, with $u\prec L+c$,
and any $t\leq\delta_2$ we have
$$\forall x\in \BB(r''), \quad T_t^-u(x)=\inf_{y\in\BB(r)}u(y)+h_t(y,x).$$
Fix  a function $u$ satisfying (a), (b), and (c) of the Lemma. We will show that $T^-_tu$ is {\rm C}$^{1,1}$ on $\BB(r'')$ for $t$ small enough (depending on $A,B$ and not on $u$), and we will compute the Lipschitz constant of the derivative of this function. Classically one shows that $T^-_tu$ is {\rm C}$^{1,1}$ by using the inverse function theorem for Lipschitz perturbation of the identity. For a change, we will do it in a (very slightly) different way using that $T^-_tu$ is Lipschitz.

Suppose $t\in \delta_1$. For $x\in \BB(r'')$ choose a point $y_x\in \BB(r)$ such that $T_tu(x)=u(y_x)+h_t(y_x,x)$. If we choose a minimizer $\gamma:[0,t]\to M$ with
$\gamma(0)=y_x, \gamma(t)=x$, and whose action is $h_t(x,y)$, we know that
$\partial L/\partial v(x,\dot\gamma(t))$ is in the upper gradient of $T_tu$  at $x$,
and $\partial L/\partial v(x,\dot\gamma(0))$ is in the lower gradient of $u$  at $y_x$.
Since $u$ is differentiable at $y_x$ we necessarily have  $\partial L/\partial v(x,\dot\gamma(0))=d_{y_x}u$. Moreover at each point $x\in \BB(r'')$ where the Lipschitz function $T_tu$ is differentiable we must have $d_xT^-_tu= \partial L/\partial v(x,\dot\gamma(t))$. Since $\gamma$ is a minimizer,
its speed curve $s\mapsto 
(\gamma(s),\dot\gamma(s))$ is an orbit of the Euler-Lagrange  flow $\phi^L_s$ associated to $L$. Since the conjugate of $\phi_s^L$ is the Hamiltonian flow $\phi^H_s$ of the Hamiltonian $H:T^*M\to\R$ associated by Fenchel duality to $L$, we obtain that at each $x$ where $T^-_tu$ is differentiable
\begin{equation*}(x,d_xT^-_tu)=\phi^H_t(y_x,d_{y_x}u).\tag{*}
\end{equation*}
Therefore $x=\pi^*\phi^H_t(y_x,d_{y_x}u)$, where $\pi^*$ is the canonical projection
from $T^*M$ to $M$. In the local coordinates that we are using $\pi^*:\BB(r)\times (\R^k)^*\to \BB(r)$ is the projection on the first factor. 
To simplify computations we use the norm $\lVert (x,p)\rVert=\max ( \lVert v\rVert_{\mathrm {euc}},  \lVert p\rVert_{\mathrm {euc}})$ on $\BB(R)\times (\R^k)^*\subset
\R^k\times (\R^k)^*$.
Let us set $\psi(s,y,p)=\pi^*(\phi^H_s(y,v))-y$. This map is {\rm C}$^1$ and is is identically $0$ when $s=0$, therefore on the compact set $\{(y,p)\in \BB(r)\times  (\R^k)^*\mid  \lVert p\rVert\leq A\}$ the Lipschitz constant $\ell(s)$ of $(y,p)\mapsto \psi(s,y,p)$ tends to $0$ as $s\to 0$.
Since $y\mapsto d_yu$ has a Lipschitz constant bounded by $B\geq1$ on $\BB(r)$, the map $y\mapsto (y,d_yu)$ has also a Lipschitz constant bounded by $B$ on $\BB(r)$. Moreover since $  \lVert d_yu\rVert_{\mathrm {euc}}$ is bounded by $A$ on $\BB(r)$, we see that on $\BB(r)$ we have
$$\pi^*\phi^H_t(y,d_{y}u)=y+\theta_{t,u}(y),$$
where the map $\theta_{t,u}$ has Lipschitz constant $\leq B\ell(t)$. Note that this $\ell(t)$ depends only on $A$ and not on $u$. Let us set $\Theta_{t,u}(y)=y+\theta_{t,u}(y)$. 
Note that 
\begin{align*}
\lVert \Theta_{t,u}(y')-\Theta_{t,u}(y)\rVert&=\lVert [y'+\theta_{t,u}(y')]-[y+\theta_{t,u}(y)]\rVert\\
&\geq \lVert y'-y\rVert-\lVert\theta_{t,u}(y')-\theta_{t,u}(y)\rVert\\
&\geq \lVert y'-y\rVert-B\ell(t)\lVert y'-y\rVert\\
&=(1-B\ell(t))\lVert y'-y\rVert.
\end{align*}
Therefore, for $t$ small enough to have $1-B\ell(t)>0$, the map $\Theta_{t,u}:\BB(r)\to
\Theta_{t,u}(\BB(r))$ is invertible and its inverse $\Theta_{t,u}^{-1}:\Theta_{t,u}(\BB(r))\to \BB(r)$ has a Lipschitz constant $\leq (1-B\ell(t))^{-1}$.
Note that Equation (*) above shows that, for every $x\in \BB(r'')$ at which $T^-_tu$ is differentiable, we can find
$y_x\in \BB(r)$ such that $x=\Theta_{t,u}(y_x)$. 
Since $T^-_tu$ is Lipschitz it is differentiable a.e, and so the  image $\Theta_{t,u}(\BB(r))$ contains a set of full Lebesgue measure in $ \BB(r'')$. The compactness of $\Theta_{t,u}(\BB(r))$ implies that this image has to contain $ \BB(r'')$. Equation (*) tells us now that at each point $x\in \BB(r'')$ where $T^-_tu$ is differentiable we have
\begin{equation*}(x,d_xT^-_tu)=\phi_t^H(\Theta_{t,u}^{-1}(x),d_{\Theta_{t,u}^{-1}(x)}u).\tag{**}.
\end{equation*}
But the right hand side above is a continuous function defined at least on $\BB(r'')$.
This implies that the Lipschitz function $T^-_tu$ is differentiable on $\BB(r'')$ and its derivative satisfies Equation (**) above. Therefore on $\BB(r'')$ the derivative $x\mapsto d_xT^-_tu$ has Lipschitz constant bounded by $L(t)B(1-B\ell(t))$, with $L(t)$ the Lipschitz constant of $(y,p)\mapsto \pi_2\phi_t^H(y,p)$ on the set $ \{(y,p)\in \BB(r)\times (\R^k)^* \mid\lVert p\rVert \leq A\}$, 
where $\pi_2:\BB(r)\times (\R^k)^*\to (\R^k)^*$ is the projection on the second factor.
Since $\phi_t^H$ is a {\rm C}$^1$ flow and $\phi^H_0$ is the identity 
we have $L(t)\to 1$, as $t\to 0$.
This finishes the proof since $\ell(t)\to 0$, and $\ell(t),L(t)$ depends only on $A$ and not on $u$.
\end{proof}
Recall that a function $f:C\to \R$, defined on the convex subset $C$ of $\R^k$ is said to be $K$-semiconvex if $x\mapsto f(x)+K\lVert x\rVert^2_{\mathrm{euc}}$ is convex on $C$. If $K\geq 0$ is fixed, for an {\sl open} convex subset $C$ of $\R^k$, the following conditions are equivalent:
\begin{itemize}
\item the function $f:C\to \R$ is $K$-semiconvex;
\item for every $x\in C$ we can find $p_x\in(\R^k)^*$ such that
$$\forall y\in C, \quad f(y)\geq p_x(y-x)-K\lVert x\rVert^2_{\mathrm{euc}};$$
\item for every $x\in C$, if $\tilde p_x\in(\R^k)^*$ is a subdifferential of $f$ at $x$,
we have
$$\forall y\in C,\quad f(y)\geq \tilde p_x(y-x)-K\lVert x\rVert^2_{\mathrm{euc}}.$$ 
\end{itemize}
It is not difficult to see that, if $C$ is an open convex subset of $\R^k$, a {\rm C}$^1$ function $f:C\to\R$, whose derivative has on $C$ a (global) Lipschitz constant $\leq B$, is $B/2$-semiconvex.

We know state the regularization property of the semi-groups $T^-_t$ and $T^+_t$.
These properties are well-known. They have been extensively exploited for viscosity solutions, see \cite{bc,bar}. For a proof in the compact case see \cite{fathibook}.
We will sketch a proof relying on the semi-concavity of $h_t$.
\begin{theorem}\label{RegLaxOl}{Suppose that $t_0>0$, that $c$  a finite constant, 
and that $(U,\varphi)$ is a {\rm C}$^\infty$ chart with $\BB(r)\subset \varphi(U)$. 
We can find a constant $K$ such that for every function $u\prec L+c$, 
and any $t\geq t_0$, the restriction $T^-_tu\circ\varphi^{-1}|\BB(r)$
(resp.~$T^+_tu\circ\varphi^{-1}|\BB(r)$) is 
$K$-semiconcave (resp.~$K$-semiconvex).}
\end{theorem}
\begin{proof} We do the proof for $T^-_t$. By Lemma \ref{globalLip}, there exists a constant $A$ such that all
functions dominated by $L+c$ have Lipschitz constant on $M$ which is $\leq A+c$.
It follows from Lemma \ref{ConvToZero} that we can find a finite  constant $B$ such that
for any $u\prec L+c$, and any $x\in M$
$$T^-_tu(x)=\inf\{u(y)+h_t(y,x)\mid y\in M, d(x,y)\leq Bt\}.$$
In particular, if $C_t$ is the compact set $\{y\in M\mid d(y, \varphi^{-1}(\BB(r)) \leq Bt\}$, we get
$$\forall x\in \varphi^{-1}(\BB(r)), \quad T^-_tu(x)=\inf_{y\in C_t}u(y)+h_t(y,x).$$
Since $h_t$ is locally semiconcave on $M\times M$ (see for example \cite[Theorem B.19]{fatfig}) and $C_t$ is a compact subset,
using standard arguments for the theory of locally semi-concave functions (again see for example
\cite[Appendix A]{fatfig}) we can find a constant $K_t$ such that
$T^-_tu\circ\varphi^{-1}|\BB(r)$ is 
$K_t$-semiconcave for every $u\prec L+c$. It remains to show that we can take $K_t$
independent of $t\geq t_0>0$. In fact, since $T^-_t$ preserves the set of functions dominated by $L+c$, we have $T^-_{t-t_0}u\prec L+c$, for any $u\prec L+c$.
Therefore, we conclude that $T^-_tu=T^-_{t_0}[T^-_{t-t_0}u]$ does also satisfy
$T^-_tu\circ\varphi^{-1}|\BB(r)$  is 
$K_{t_0}$-semiconcave.
\end{proof}
Next we show that $T^-_t$ preserve semi-convexity, for small time $t$.
\begin{lemma}\label{SemiConvexPreserved}{Suppose that $(U,\varphi)$ is a 
{\rm C}$^\infty$ chart on $M$, and $\BB(r)\subset \varphi (U)$. For any $r'<r$, 
any finite number $A\geq 0$, any finite number $K\geq 1/2$, and finite $\epsilon>0$, 
we can find $\delta>0$ such that 
for any function $u:\BB(r)\to\R$ satisfying
\begin{itemize}
\item[{\rm (i)}] the function $u$  has Lipschitz constant $\leq A$  on  $\BB(r)$;
\item[{\rm (ii)}]  the function $u$ is $K$-semiconvex;
\end{itemize}
and any $t\leq \delta$, the function $T^{-,\varphi}_tu:\BB(r')\to\R$ defined by
$$T^{-,\varphi}_tu(x)=\inf_{y\in\BB(r)}u(y)+h_t(\varphi^{-1}(y),\varphi^{-1}(x))$$
is $(K+\epsilon)$-semiconvex in $\BB(r')$.}
\end{lemma}
\begin{proof} As in the previous proof we will assume that $\varphi$ is the ``identity''.
We also choose $r''$ and $r'''$ such that $r'<r''<r'''<r$. 
Consider the family of function $v_{\alpha,x,p}:\BB(r)\to\R$, where 
$\alpha\in \R,x\in\BB(r)$ and $p\in(\R^k)^*$, with $\lVert p\rVert_{\mathrm{euc}}\leq A$,
defined by
$$v_{\alpha,x,p}(y)=\alpha+p(y-x)-K\lVert y-x\rVert_{\mathrm{euc}}^2.$$
It is not difficult to see that the derivative of $v_{\alpha,x,p}$ has, on $\BB(r)$, a Lipschitz constant $\leq 2K$, and that this derivative is bounded in norm by $A+4Kr$.
Since $2K\geq 1$, we can apply Lemma \ref{C11Conservee} and find $\delta>0$ such that
$T^{-,\varphi}_tv_{\alpha,x,p}$ is {\rm C}$^{1,1}$ on $\BB(r'')$ with a Lipschitz constant for its derivative $\leq 2K+2\epsilon$
for any $t \leq \d$. In particular any such function $T^{-,\varphi}_tv_{\alpha,x,p}$ is $(K+\epsilon)$-semiconvex on $\BB(r'')$.

Taking   $\delta>0$ smaller if necessary, we can assume that, for every $u$ satisfying condition (a) of the lemma,
every $t\leq \delta$, and every $x\in \BB(r'')$, we can find $y_x\in\BB(r''')$ such that
$$T^{-,\varphi}_tu(x)=u(y_x)+h_t(y_x,x).$$
If we pick up a minimizer $\gamma:[0,t]\to M$ with $\gamma(0)=y_x$ and $\gamma(t)=x$, we know that $\tilde p_x=\partial L/\partial v(y_x,\dot\gamma(0))$ is a subdifferential of 
$u$ at $y_x$, and also
\begin{equation*} 
\pi^*\phi^H_t(y_x,\tilde p_x)=x.\tag{***}
\end{equation*}
Since $u$ is $K$-semiconvex on $\BB(r)$ and $\tilde p_x$ is in the subdifferential of $u$ at $y_x$,
we have 
$$\forall y\in\BB(r), \quad u(y)\geq 
u(y_x)+\tilde p_x-K\lVert y-y_x\rVert_{\mathrm{euc}}^2=v_{u(y_x),y_x,\tilde p_x}(y).$$
Set $v=v_{u(y_x),y_x,\tilde p_x}$ to simplify notation. From the inequality above we get
\begin{equation*}
T^{-,\varphi}_tu\geq T^{-,\varphi}_tv.\tag{$\diamond$}
\end{equation*}
 We also know that $T^{-,\varphi}_tv$ 
is {\rm C}$^{1,1}$ and $(K+\epsilon)$-semiconvex. We know show that
$T^{-,\varphi}_tu$ and $T^{-,\varphi}_tv$ take the same value at $x$. By the proof
of the previous lemma we know that 
$T^{-,\varphi}_tv(x)=v(y'_x)+h_t(y'_x,x)$, where $y'_x$ is the only point $y\in \BB(r)$
such that $\pi^*\phi^H_t(y,d_yv)=x$. But $d_{y_x}v=\tilde p_x$, therefore by Equation (***) we obtain $y'_x=y_x$. Since we also have $v(y_x)=u(y_x)$, we conclude
that $T^{-,\varphi}_tv(x)=v(y_x)+h_t(y_x,x)=u(y_x)+h_t(y_x,x)=T^{-,\varphi}_tu(x)$.
Since $v$ is $(K+\epsilon)$-semiconvex on $\BB(r'')$, for every $y\in \BB(r'')$ we have
$$T^{-,\varphi}_tv(y)\geq T^{-,\varphi}_tv(x)+p_x(y-x)-(K+\epsilon)\lVert y-x\rVert_{\mathrm{euc}}^2,$$
where $p_x$ is the derivative at $x$ of the {\rm C}$^{1,1}$ function $T^{-,\varphi}_tv$.
Therefore by Equation ($\diamond$) we obtain
$$\forall y\in\BB(r''),\quad T^{-,\varphi}_tu(y)\geq T^{-,\varphi}_tu(x)+p_x(y-x)-(K+\epsilon)\lVert y-x\rVert_{\mathrm{euc}}^2.$$
Since $x$ was an arbitrary point in $\BB(r'')$, this finishes the proof.
\end{proof}
Before giving the proof of the theorem, we also notice that since $L$ is
uniformly superlinear in the fibers, there exists a finite constant
$C(K')$ such that
$$
\forall (x,v) \in TM, \quad L(x,v) \geq 2K' \|v\|_x + C(K').
$$
From that, we deduce that for every $t>0$,
\begin{equation}
\label{16janv0}
\forall x,y \in M, \quad h_t(x,y), h_t(y,x) \geq 2 K' d(x,y) + C(K') t.
\end{equation}
The previous two lemmas are also true if we replace $T^{-,\varphi}_tu$ by
$$T^{+,\varphi}_tu(x)=\sup_{y\in\BB(r)}u(y)-h_t(x,y),$$
and also replace semi-convexity in the second lemma by semi-concavity.

\begin{proof}[Proof of Theorem \ref{THMexistC11}] We choose a countable family of 
{\rm C}$^{\infty}$ charts $(U_n,\varphi_n)_{n\geq 1}$ on $M$ such that
 $\varphi_n(U_n)=\R^k$, and $M=\cup_{n\geq 0}\varphi_n^{-1}(\INT\BB(1))$. 
 
 Fix a $c\in\R$. We know that any function $u:M\to\R$ dominated by $L+c$ is Lipschitz with Lipschitz constant $\leq A(0)+c$. Therefore, for each integer $n\geq 1$, we can find a finite constant $A_n$ such that, for every $u:M\to\R$ dominated by $L+c$, the function
 $u\circ\varphi^{-1}_n$ has on $\BB(2)$ a Lipschitz constant $\leq A_n$ for the canonical Euclidean norm on $\R^k$. We will construct by induction a sequence $B_n\in [1,+\infty[$
 and two sequence of $>0$ numbers $t_n^-,t^+_n$ such that if we define, for $u:M\to\R$, the function $S_n(u)$ on $m$ by 
 $$S_n(u)=T^-_{t_n^-}T^+_{t^+_n}T^-_{t_{n-1}^-}T^+_{t^+_{n-1}}
 \cdots T^-_{t_1^-}T^+_{t^+_1}(u),$$
with $S_0$ the identity, then for every $u\prec L+c$ defined on the whole $M$, and 
 every $k=1,\dots, n$, we have
 \begin{itemize}
 \item[(i)] the supremum $\sup_{x\in M}\lvert S_n(u)(x)-S_{n-1}(u)(x)\rvert$ 
 is less than  $1/2^n$;
\item[(ii)] the function $S_n(u)\circ\varphi_k^{-1}$ is {\rm C}$^{1,1}$ on $\BB(1+2^{-n})$;
\item[(iii)] the function  $S_n(u)\circ\varphi_k^{-1}$ has on $\BB(2)$  Lipschitz constant $
\leq A_k$;
\item[(iv)] the derivative of $S_n(u)\circ\varphi_k^{-1}$  on $\BB(1+2^{-n})$ has Lipschitz constant $\leq B_k+1-2^{-n}$.
\end{itemize}
Note that, since $T^-_t$ and $T^+_t$ do preserve
 functions dominated by $L+c$ on  $M$, we will have $S_n(u)\prec L+c$, and condition (iii) above will be satisfied for any  choice of  $t^+_n,t^-_n$. 
 
 Suppose that $S_{n}$ has been constructed.  We first pick $t^+_{n+1}$. It follows from Lemma \ref{ConvToZero} that there exists $\delta_1$ such that for every $u\prec L+c$, and $t\in[0, \delta_1]$, we have 
 $$\sup_{x\in M}\lvert T^+_t(u)(x)-u(x)\rvert\leq 1/2^{n+2}.$$
 Given that (i), (ii), (iii) and (iv) are verified, we can apply the version of 
 Lemma \ref{C11Conservee}
 for $T^+_t$ to the finite set of charts $(U_k,\varphi_k),k=1,\dots,n$, 
 with ball $\BB(1+2^{-n})$ constants $A_k$ and $B_k+1-2^{-n}$, to find $\delta_2>0$ 
 such that for every $t\in[0,\delta_2]$, and every $u\prec L+c$, 
 the function $T^+_tS_n(u)\circ\varphi_k^{-1}$ is {\rm C}$^{1,1}$   on 
 $\BB(1+2^{-n}-2^{-(n+2)})$ with Lipschitz constant of its derivative 
 $\leq B_k+1-2^{-n}+2^{-(n+2)}$. 
 Let us now fix $t^+_{n+1}>0$ with $t^+_{n+1}\leq\min(\delta_1,\delta_2)$.
 Since $t^+_{n+1}>0$, we know by Theorem  \ref{RegLaxOl} 
 that there exists a finite constant
 $\tilde B_{n+1}$ such that for every $u\prec L+c$, the function 
 $T^+_{t^+_{n+1}}S(u)\circ \varphi_{n+1}^{-1}$ is $\tilde B_{n+1}$-semiconvex on the ball
 $\BB(2)$. Therefore, for every $u\prec L+c$, we have
 \begin{itemize}
 \item[(a)] the supremum 
 $\sup_{x\in M}\lvert T^+_{t^+_{n+1}}S_n(u)(x)-S_{n}(u)(x)\rvert$ 
 is less than  $1/2^{n+2}$;
\item[(b)] the function $T^+_{t^+_{n+1}}S_n(u)\circ\varphi_k^{-1}$ is {\rm C}$^{1,1}$ on $\BB(1+2^{-n}-2^{-(n+2)})$, for $k=1,\dots,n$;
\item[(c)] the function  $T^+_{t^+_{n+1}}S_n(u)\circ\varphi_k^{-1}$ has on $\BB(2)$  Lipschitz constant
$\leq A_k$, for $k=1,\dots,n+1$;
\item[(d)] the derivative of $T^+_{t^+_{n+1}}S_n(u)\circ\varphi_k^{-1}$  on $\BB(1+2^{-n})$ has Lipschitz constant $\leq B_k+1-2^{-n}+2^{-(n+2)}$;
\item[(e)] the function $T^+_{t^+_{n+1}}S_n(u)\circ\varphi_{n+1}^{-1}$ is 
$\tilde B_{n+1}$-semiconvex on the ball $\BB(2)$.
\end{itemize}
We first pick $t^-_{n+1}$. It follows from Lemma \ref{ConvToZero} that 
there exists $\delta'_1$ such that for every $u\prec L+c$, and $t\in[0,\delta'_1]$,
we have 
$$\sup_{x\in M}\lvert T ^-_t(u)(x)-u(x)\rvert\leq 1/2^{n+2}.$$
Given that (b), (c), (d) are verified, we can apply  Lemma \ref{C11Conservee}
to the finite set of charts $(U_k,\varphi_k), k=1,\dots,n$, with ball 
$\BB(1+2^{-n}-2^{-(n+2)})$, constants $A_k$ and $B_k+1-2^{-n}+2^{-(n+2)}$, to find 
$\delta'_2>0$ such that for every $t\in[0,\delta'_2]$, and every $u\prec L+c$, 
the function $T^-_tT^+_{t^+_{n+1}}S_n(u)\circ\varphi_k^{-1}$ is {\rm C}$^{1,1}$  
on $\BB(1+2^{-(n+1)})=\BB(1+2^{-n}-2^{-(n+2)}-2^{-(n+2)})$ with Lipschitz constant
of its derivative $\leq B_k+1-2^{-n}+2^{-(n+2)}+2^{-(n+2)}=B_k+1-2^{-(n+1)}$. 

By condition (e) above, we can also apply Lemma \ref{SemiConvexPreserved}
 to find $\delta'_3$ 
such that for every $t\in [0,\delta'_3]$, and each $u\prec L+c$, the function
 $T^-_tT^+_{t^+_{n+1}}S_n(u)\circ\varphi_{n+1}^{-1}$ is $(\tilde B_{n+1}+1)$-semiconvex 
 on $\BB(1+2^{-(n+1)})$.
Let us now fix $t^-_{n+1}>0$ with $t^-_{n+1}\leq\min(\delta'_1,\delta'_2,\delta'_3)$.
Since $t^-_{n+1}>0$, we know by Theorem
 \ref{RegLaxOl} that there exists a finite constant
$\hat B_{n+1}$ such that, for every $u\prec L+c$, the function 
$T^-_{t^-_{n+1}}T^+_{t^+_{n+1}}S_n(u)\circ \varphi_{n+1}^{-1}$ is $\hat B_{n+1}$-semiconcave on the ball
$\BB(1+2^{-(n+1)})$. Hence, if we set 
$B_{n+1}=2\max\{\tilde B_{n+1}+1,\hat B_{n+1}\}\geq 1$, 
for every $u\prec L+c$ the function 
$T^+_{t^-_{n+1}}T^+_{t^+_{n+1}}S_n(u)\circ\varphi_{n+1}^{-1}$ 
is both $ B_{n+1}/2$-semiconvex and $ B_{n+1}/2$-semiconcave on
$\BB(1+2^{-(n+1)})$. It is therefore {\rm C}$^{1,1}$ on  $\BB(1+2^{-(n+1)})$, with a derivative with Lipschitz constant $\leq B_{n+1}$.
It is not difficult now to verify that with this choice of $t^+_{n+1},t^-_{n+1}$, the operator $S_{n+1}$ satisfies the required conditions (i), (ii), (iii), and (iv).
\end{proof}

\end{document}